%% file: affine.tex
\documentclass[11pt,a4paper, twoside]{amsart}
\usepackage{geometry}
\usepackage[utf8]{inputenc}
\usepackage[english]{babel}
\usepackage[T1]{fontenc}
\usepackage{lmodern}

\usepackage{amsmath}
\usepackage{amssymb}
\usepackage{amsthm}

\usepackage{mathtools}
\usepackage{stmaryrd}

\usepackage{wrapfig}
\usepackage{subcaption}

\usepackage{enumitem}
\usepackage{verbatim}
\usepackage{wrapfig}
\usepackage{makecell}
\usepackage{xcolor}

\usepackage{hyperref}
\hypersetup{colorlinks,linkcolor={blue!75!black},citecolor={blue!75!black},urlcolor={blue!75!black}}

\usepackage[backend=biber,maxnames=50]{biblatex}
\usepackage{tikz}

\DeclareMathOperator*{\id}{id}

\DeclareMathOperator*{\linh}{span}
\DeclareMathOperator*{\sgn}{sgn}
\DeclareMathOperator*{\argmax}{arg\,max}
\DeclareMathOperator*{\argmin}{arg\,min}

\DeclareMathOperator*{\conv}{conv}

\DeclareMathOperator*{\aff}{aff}
\DeclareMathOperator*{\pos}{pos}
\DeclareMathOperator*{\bd}{bd}
\DeclareMathOperator*{\cl}{cl}
\DeclareMathOperator*{\relbd}{relbd}

\newcommand{\MID}{\,\middle|\,}
\DeclarePairedDelimiter\scpr{\langle}{\rangle}

\DeclarePairedDelimiter{\norm}{\lVert}{\rVert}
\newcommand{\dotnorm}{\norm{\cdotbox}}
\newcommand{\NORM}[1]{\left\|#1\right\|}
\newcommand{\SCPR}[1]{\left\langle #1\right\rangle}

\DeclareMathOperator*{\Subl}{Subl}

\newcommand{\R}{\mathbb{R}}

\newcommand{\N}{\mathbb{N}}

\mathtoolsset{showonlyrefs}

\theoremstyle{theorem}

\newtheorem{theorem}{Theorem}[section] %
\newtheorem{corollary}[theorem]{Corollary}
\newtheorem{proposition}[theorem]{Proposition}
\newtheorem{lemma}[theorem]{Lemma}

\theoremstyle{definition}

\newtheorem{problem}[theorem]{Problem}
\newtheorem{definition}[theorem]{Definition}
\newtheorem{construction}[theorem]{Construction}

\theoremstyle{remark}
\newtheorem{remark}[theorem]{Remark}

\numberwithin{equation}{section}

\DeclareMathOperator*{\vertices}{vert}
\DeclareMathOperator*{\interior}{int}
\DeclareMathOperator*{\relint}{relint}
\DeclareMathOperator*{\vol}{vol}

\newcommand{\volt}[1]{{\textstyle\vol_{#1}}}

\DeclareMathOperator*{\tr}{tr}

\DeclareRobustCommand{\One}{\text{\usefont{U}{bbold}{m}{n}1}}

\newcommand{\transp}{{\mkern-1.5mu\mathsf{T}}}
\newcommand{\Kn}{\mathcal{K}^n}
\newcommand{\Knn}{\mathcal{K}_n^n}

\newcommand{\Kno}{\mathcal{K}_{o}^n}

\newcommand{\E}{\mathbb{E}}

\newcommand{\Lpm}{\mathsf{L}^\pm}

\DeclareMathOperator*{\C}{C}

\DeclareMathOperator*{\Sym}{Sym}
\newcommand{\GRS}{\mathsf{G}}
\DeclareMathOperator*{\Orthogonal}{O}

\DeclareMathOperator*{\GL}{GL}
\newcommand{\On}{\Orthogonal(n)}

\newcommand{\GLnR}{\GL(n,\R)}
\newcommand{\Symn}{\textstyle \Sym_n}
\newcommand{\sphere}{\mathbb{S}^{n-1}}

\newcommand{\dint}{\,\mathrm{d}}
\newcommand{\ddt}{\frac{\mathrm{d}}{\mathrm{d}t}}
\newcommand{\ddtn}[1]{\frac{\mathrm{d}^{#1}}{\mathrm{d}t^{#1}}}

\newcommand{\fwa}{\mathsf{F}}

\newcommand{\cdotbox}{\hspace{0.4mm}\cdot\hspace{0.4mm}}

\AtEveryBibitem{%
	\clearfield{number}}

\title{Shadow systems, decomposability and isotropic constants}
\author{Christian Kipp}
\keywords{shadow system, decomposability, isotropic constant, polytope}

\address{Technion -- Israel Institute of Technology, Department of Mathematics, Haifa 32000, Israel}
\email{ckipp@campus.technion.ac.il}
\begin{document}
	
	\begin{abstract}
		We study necessary conditions for local maximizers of the isotropic constant that are related to notions of decomposability. Our main result asserts that the polar body of a local maximizer of the isotropic constant can only have few Minkowski summands; more precisely, its dimension of decomposability is at most $\frac12(n^2+3n)$. Using a similar proof strategy, a result by Campi, Colesanti and Gronchi concerning RS-decomposability is extended to a larger class of shadow systems. We discuss the po\-ly\-to\-pal case, which turns out to have connections to (affine) rigidity theory, and investigate how the bound on the maximal number of irredundant summands can be improved if we restrict our attention to convex bodies with certain symmetries.
	\end{abstract}
	
	\maketitle
	
	\setcounter{tocdepth}{1}
	
	\tableofcontents
	
	\section{Introduction} \label{sect_introduction}

	Let $\Kn$, $\Knn$ and $\Kno$ denote the nested spaces of convex bodies in $\R^n$, full-dimensional convex bodies in $\R^n$, and full-dimensional convex bodies in $\R^n$ whose interior contains the origin, respectively. The \emph{centroid} $c_K$ and the \emph{covariance matrix} $A_K$ of a (full-dimensional) convex body $K \in \Knn$ are defined by $c_K \coloneqq \frac1{\vol (K)} \int_K x \dint x$ and
	\begin{equation} \label{eq_def_L_K}
		A_K \coloneqq \frac1{\vol (K)}\int_K (x-c_K) (x-c_K)^\transp \dint x \in \R^{n \times n},
	\end{equation}
	respectively. The \emph{isotropic constant} $L_K$ of $K$ is given by
	\begin{equation}
		L_K^{2n} = \frac{\det A_K}{\vol(K)^2}.
	\end{equation}
	Because $L_K$ is invariant under affine transformations of $K$, it is often enough to study the case where $K$ is \emph{isotropic} in the sense that $c_K$ is the origin and $A_K$ is the identity matrix.

	The famous \emph{slicing problem} in asymptotic geometric analysis is equivalent to the question whether $L_K$ is bounded from above by a universal constant. The problem was first studied in the 1980's by Bourgain \cite{bou86}, who showed in \cite{bou91} that $L_K \leq C \cdot \log n \cdot n^{\frac14}$. The log-factor was removed in 2006 by Klartag \cite{kla06}. More recently, Eldan's method of stochastic localization \cite{eld13} has led to spectacular progress on this notoriously hard problem. Starting with a breakthrough result by Chen \cite{che21}, a series of improvements of the upper bound on $L_K$ was achieved \cite{kl22,jls22} through progress on the related \emph{KLS conjecture}. Until December 2024, the best known bound, due to Klartag \cite{kla23}, was $L_K \leq C\cdot\sqrt{\log n}$. In December 2024, Guan uploaded a preprint \cite{gua24} to \texttt{arxiv.org}, whose title ``A note on Bourgain’s slicing problem'' turned out to be quite an understatement. Exactly one week later, Klartag and Lehec replied with a further preprint \cite{kl24}, revealing that a bound from Guan's preprint \cite[Lem.~2.1]{gua24} provided the ``missing link'' in a proof strategy they had developed a few years earlier, and showing how this leads to an affirmative solution of the slicing problem. Since the slicing problem has numerous deep connections to other fundamental questions in high-dimensional convex geometry (see \cite{km22,bgvv16,gar24}), the significance of this breakthrough can hardly be overstated.
	
	A natural follow-up question is to ask for a sharp upper bound on $L_K$ or, equivalently, for the global maximizers of $K\mapsto L_K$ in a given dimension $n \in \N$. With regard to this problem, a non-asymptotic version of the isotropic constant conjecture asserts that every convex body $K \in \Knn$ satisfies
	\begin{equation} \label{eq:intro:strong_isotropic_constant_conjecture}
		L_K \leq L_{\Delta_n} = \frac{\sqrt[n]{n!}}{(n+1)^\frac{n+1}{2n} \cdot \sqrt{n+2}},
	\end{equation}
	where $\Delta_n$ is an $n$-dimensional simplex. Since $L_{\Delta_n} \leq C$, this statement is sometimes called the \emph{strong isotropic constant conjecture}. As shown by Klartag \cite{kla18}, the strong isotropic constant conjecture implies the asymmetric version of the well-known Mahler conjecture.
	
	In dimension $n=2$, \eqref{eq:intro:strong_isotropic_constant_conjecture} follows from a result by Campi, Colesanti and Gronchi \cite{ccg99} and it is known that triangles are the only local maximizers of $K \mapsto L_K$, as shown by Saroglou \cite{sar10}. The affirmative resolution of the slicing problem notwithstanding, not much is known about local maximizers of $K \mapsto L_K$ in dimensions $n \geq 3$. In the following, we collect some known facts about such local maximizers.
	
	Campi, Colesanti and Gronchi \cite{ccg99} proved that a convex body $K\in \Knn$ whose boundary contains an open set $U$ which is $C^2$ with positive principal curvatures cannot be a local maximizer of the isotropic constant. Meyer and Reisner \cite{mr16} strengthened this result, showing the following.
	
	\begin{theorem}[Meyer-Reisner] \label{thm_mr}
		If the boundary of $K \in \Kno$ contains a %
		relatively open set $U$ that is strictly convex, then $K$ cannot be a local maximizer of the isotropic constant.
	\end{theorem}
	
	Actually, Meyer and Reisner showed a much stronger statement in \cite{mr16}, namely, that even a single point $x \in \bd K$ of strict convexity and with positive generalized Gauss curvature suffices to rule out that $K$ is a local maximizer. With regard to the polytopal case, Rademacher proved in \cite{rad16} that a simplicial polytope $P$ that maximizes the isotropic constant has to be a simplex. In \cite{kip25}, this statement was strengthened as follows.
	
	\begin{theorem} \label{thm_simplicial_vertex}
		A polytope $P$ has to be a simplex if it is a local maximizer of $K \mapsto L_K$ and it has a simplicial vertex, i.e., a vertex $v$ such that every facet $F \subset P$ with $v \in F$ is a simplex.
	\end{theorem}
	
	Moreover, it was shown in \cite{kip25} that a zonotope that maximizes the isotropic constant in the space of centrally symmetric convex bodies and that has a cubical zone (a zone that contains only parallelepipeds) must be a cube \cite[Thm.~1.4]{kip25}.
		
	Intuitively speaking, a common theme of the aforementioned results is that a local maximizer of $K \mapsto L_K$ cannot be too ``flexible'' or, equivalently, has to be sufficiently ``rigid'' to rule out modifications that provably increase the isotropic constant. A goal of the present article is to substantiate this intuition.
	
	The overarching theme of our investigation will be notions of decomposability of a convex body.  %
	In this article, a notion of decomposability corresponds to an embedding of a class of convex bodies $\mathcal{A} \subset \Kn$ into a function space $V$, which assigns to each convex body $K \in \mathcal{A}$ a function $f_K$ (e.g., its support function $h_K$ or its gauge function $\dotnorm_K$). In the cases under consideration, the image of the embedding $A\coloneqq \{f_K \mid K \in \mathcal{A}\}$ will be convex. Given $f_K$, we consider the lineality space of the support cone of $A$ at $f_K$, which is given by
	\begin{equation}
		\Lpm(f_K,A) \coloneqq \{g \in V \mid \exists \varepsilon>0 \colon f_K+\varepsilon g,f_K-\varepsilon g \in A\}.
	\end{equation}
	The dimension of the vector space $\Lpm(f_K,A)$ can be interpreted as a quantitative measure of the decomposability of $K$, i.e., the ``number of substantially different ways'' in which we can write $f_K$ as a convex combination $f_K = \frac12 (f_L+f_M)$ with $L,M \in \mathcal{A}$. Our goal is to prove necessary conditions for local maximizers of $K\mapsto L_K$ of the form:
	\begin{quotation}
		If $\dim\Lpm(f_K,A)$ is ``too large'', then $K$ cannot be a local maximizer of the isotropic constant.
	\end{quotation}
	
	We first study the notion of RS-decomposability, due to Campi, Colesanti and Gronchi \cite{ccg99}. This notion provides a versatile tool for determining structural properties of local maximizers of $K\mapsto L_K$; it is used to prove the mentioned results in \cite{ccg99,rad16,kip25} and can be used to prove Theorem \ref{thm_mr}. Let $K \in \Knn$ and $u \in \sphere$.  We consider the set of full-dimensional convex bodies that have the same X-ray \cite[Ch.~2]{gar06} in direction $u$ as $K$:
	\begin{equation}
		\mathcal{K}_{K,u} \coloneqq \{\tilde{K} \in \Knn \mid \volt{1}[(x+\linh\{u\})\cap \tilde{K}]=\volt{1}[(x+\linh\{u\})\cap K]\text{ for all } x\in \R^n\}.
	\end{equation}
	Let $\pi_{u^\bot}$ denote the orthogonal projection onto the orthogonal complement $u^\bot$ of the line $\linh \{u\}$. For every $L \in \mathcal{K}_{K,u}$, there exist two convex functions $f_{L,u},g_{L,u} \colon \pi_{u^\bot}(K) \rightarrow \R$ with
	\begin{equation} \label{eq_def_f_and_g}
		L = \left\{x \in \pi_{u^\bot}^{-1}(K)  \MID f_{L,u}(\pi_{u^\bot}(x))\leq \scpr{x,u} \leq -g_{L,u}(\pi_{u^\bot}(x))\right\}.
	\end{equation}
	The process of assigning the functions $f_{K,u}$ and $g_{K,u}$ to $K$ can be interpreted as embedding the set $\mathcal{K}_{K,u}$ into the vector space of %
	functions $(\pi_{u^\bot}(K) \times \{0,1\}) \rightarrow \R$. %
	This embedding entails a notion of decomposability, which is the content of the following definition.
	\begin{definition} \label{def_RS_decomposable}
		A convex body $K \in \Knn$ is called \emph{RS-decomposable in  direction $u \in \sphere$} if there exists a non-affine \emph{speed function} $\beta \colon \pi_{u^\bot}(K) \rightarrow \R$ such that the functions
		\begin{equation}
			f_{K,u}+\beta,\quad f_{K,u}-\beta,\quad g_{K,u}+\beta,\quad g_{K,u}-\beta
		\end{equation}
		are all convex.
	\end{definition}
	Geometrically speaking, a convex body $K \in \Knn$ is RS-decomposable if there exists a direction $u$ and two convex bodies $K_1,K_2 \in \mathcal{K}_{K,u}$, none of which is an affine image of $K$, such that $f_{K,u}=\frac12(f_{K_1,u}+f_{K_2,u})$ and $g_{K,u}=\frac12(g_{K_1,u}+g_{K_2,u})$. %
	The following powerful result from \cite{ccg99} asserts that a convex body $K$ that admits such a decomposition cannot be a local maximizer of the isotropic constant.
	
	\begin{theorem}[Campi-Colesanti-Gronchi] \label{thm_ccg}
		Let $K \in \Knn$ be RS-decomposable in direction $u \in \sphere$ with speed function $\beta$. For $t \in [-1,1]$, let $K_t$ be the unique convex body in $\mathcal{K}_{K,u}$ with
		\begin{equation}
			f_{K_t,u} = f_{K,u} +t\cdot \beta \quad \text{and} \quad  g_{K_t,u} = g_{K,u} -t\cdot \beta.
		\end{equation}
		Then $t \mapsto L_{K_t}^{2n}$ is strictly convex. In particular, $K$ is not a local maximizer of the isotropic constant.
	\end{theorem}
	Following \cite{ccg99}, we call a one-parameter family of convex bodies $t \mapsto K_t$ as in Theorem \ref{thm_ccg} an \emph{RS-movement} of $K$. In Section \ref{sect_shadow_systems}, we give a new proof of Theorem \ref{thm_ccg}, based on a computation of the second derivative of $K \mapsto L_K$ (which is done in Section \ref{sect_second_derivative}). 
	Essentially as a byproduct of our new proof for Theorem \ref{thm_ccg}, we obtain a new necessary condition for local maximizers of $K \mapsto L_K$ involving certain variations of $K$ which we call \emph{generalized RS-movements}. Because the definition of generalized RS-movements and the corresponding necessary condition for local maximizers are somewhat technical, we defer them to Section \ref{sect_shadow_systems}.
	
	The second notion of decomposability that we study results from identifying convex bodies with their gauge functions or, equivalently, with the support functions of their polar bodies. We recall that a function $h \colon \R^n \rightarrow \R$ is called \emph{positively homogeneous} if we have $h(\lambda x) = \lambda h(x)$ for all $\lambda \geq 0$, $x \in \R^n$.
	A function is called \emph{sublinear} if it is positively homogeneous and convex. Every sublinear function is continuous. We denote the convex cone of all sublinear functions on $\R^n$ by $\Subl(\R^n)$. Every convex body $K \in \Kn$ induces a sublinear function $h_K$ via
	\begin{equation}
		h_K(x) = \sup_{y \in K} \scpr{x,y} \quad \text{for } x \in \R^n.
	\end{equation}
	The function $h_K$ is called the \emph{support function} of $K$. Since every sublinear function $\R \rightarrow \R$ is the support function of a unique convex body, there is a one-to-one correspondence between convex bodies and sublinear functions \cite[Sect.~1.7]{sch13}. Support functions are the natural functional representatives of convex bodies in the context of the Brunn-Minkowski theory because the vector addition given by the embedding $\Kn \rightarrow \Subl(\R^n)$, $K \mapsto h_K$ is the Minkowski sum $K+L \coloneqq \{x+y \mid x \in K, y \in L\}$. A convex body $L\in \Kn$ is called a \emph{Minkowski summand} of $K\in \Kn$ if there exists a decomposition
	\begin{equation}
		K = L+M
	\end{equation}
	with $M \in \Kn$. We denote the convex cone of convex bodies that are homothetic to (i.e., result from scaling and translating) a Minkowski summand of $K$ by $\mathcal{S}(K)$ and call the dimension of this cone the \emph{dimension of decomposability} of $K$.%
	
	Now let $K \in \Kno$. The \emph{gauge function} of $K$ is the sublinear function
	\begin{equation}
		\norm{x}_K \coloneqq  \min \{\lambda \geq 0 \mid  x \in \lambda K\}.
	\end{equation}
	Equivalently, $\dotnorm_K$ is the unique asymmetric norm on $\R^n$ whose unit ball is $K$. As a sublinear function, $\dotnorm_K$ is also the support function of a convex body $K^\circ$, which is called the \emph{polar body} of $K$. The map $K \mapsto K^\circ$ defines an involution on $\Kno$. If we identify convex bodies with their gauge functions, the analogue of an RS-movement is a linear interpolant of gauge functions, i.e., a family $(K_t)_{t \in [0,1]}$ given by an equation of the form $\dotnorm_{K_t} = (1-t) \dotnorm_{K_{0}} +t \dotnorm_{K_{1}}$.
	
	In Section \ref{sect_minkowski_decomposability}, we prove our main result, which asserts that the polar body $K^\circ$ of a local maximizer $K$ of the isotropic constant has few Minkowski summands.
	
	\begin{theorem} \label{thm_minkowski_summands}
		Let $K \in \Kno$ be a local maximizer of the isotropic constant. Then the dimension of decomposability of $K^\circ$ satisfies
		\begin{equation}
			\dim \mathcal{S}(K^\circ) \leq \frac{n^2+3n}{2}.
		\end{equation}
	\end{theorem}
	
	Recalling Theorem \ref{thm_mr} and Theorem \ref{thm_simplicial_vertex} and the related results from the beginning of this section, let us shortly discuss how Theorem \ref{thm_minkowski_summands} illustrates a common underlying geometric phenomenon of these facts. We first note that if the boundary of a local maximizer $K$ contains a strictly convex open set $U$, then \cite[Lem.~5]{mr16} in combination with a result by Sallee \cite[Lem.~4.1]{sal72} implies that $\dim \mathcal{S}(K^\circ)=\infty$. In this sense, our result provides an alternative proof for Theorem \ref{thm_mr}. With regard to Theorem \ref{thm_simplicial_vertex}, it is known that if $P$ is a polytope, %
	then $\dim\mathcal{S}(P^\circ)$ can be expressed in terms of the affine dependences of the vertices on each facet $F \subset P$. If we partition the boundary complex of $P$ into connected components by removing all facets of $P$ that are simplices (but keeping their vertices), then Theorem \ref{thm_minkowski_summands} implies that there cannot be ``too many'' of these connected components. This phenomenon will be discussed in more detail in Section \ref{sect_the_polytopal_case}, where we analyze the polytopal case of Theorem \ref{thm_minkowski_summands}. Concerning the class of zonotopes, we note in passing that $\dim\mathcal{S}(Z^\circ)$ is related to the number of cubical zones of a zonotope $Z \in \Kno$ (this follows from the correspondence between Minkowski summands of the polar body and facewise affine maps, which is discussed in Section \ref{sect_the_polytopal_case}).
	
	In Section \ref{sect_symmetric_convex_bodies}, we consider classes of convex bodies with certain symmetries. In particular, we consider centrally symmetric, unconditional and 1-symmetric convex bodies. It turns out that if we restrict our attention to convex bodies with symmetries, then the upper bound on the dimension of decomposability of the polar body can be improved.
	
	The article ends with some concluding remarks in Section \ref{sect_concluding_remarks}, where we discuss whether the bound in Theorem \ref{thm_minkowski_summands} can be improved in general.
	
	\section{Derivatives of the isotropic constant} \label{sect_second_derivative}

	The main goal of this section is to compute the second derivative of $K \mapsto L_K$. As a starting point, we reproduce Rademacher's computation of the first derivative. Rademacher's computation \cite{rad16} yields the derivative of $P_t \mapsto L_{P_t}$ as a facet of a simplicial polytope $P_0$ is hinged around one of its ridges. In order to state the result slightly more generally, we recall a notion of weak differentiability of convex-body-valued maps that was introduced and studied by Pflug and Weisshaupt \cite{pfl96,wei01,pw05}. Here and in the following, we use the term \emph{signed measure} to denote a linear combination of Borel probability measures on $\R^n$ and write $C(\R^n)$ for the vector space of continuous functions $\R^n \rightarrow \R$, equipped with the supremum norm.
	
	\begin{definition} \label{def_weak_derivative}
		A map $[t_1,t_2] \rightarrow \Knn$, $t \mapsto K_t$ is called \emph{weakly differentiable} at $t^* \in [t_1,t_2]$ if there exists a continuous linear functional $\mu \colon C(\R^n) \rightarrow \R$ with
		\begin{equation} \label{eq_def_weak_derivative}
			\left.\ddt \int_{K_t} g(x) \dint x \right|_{t=t^*} = \mu(g) \quad \text{for all } g \in C(\R^n).
		\end{equation}
		By the Riesz representation theorem, the linear functional $\mu$ can be identified with a signed measure (which happens to be supported on $\bd K_{t^*}$). The linear functional $\mu$ is called the \emph{weak derivative} of $t \mapsto K_t$ at $t^*$.
	\end{definition}
	
	We note that a map $[t_1,t_2] \rightarrow \Knn$ can be weakly differentiable at the boundary points $t_1$ and $t_2$ of its domain. In these cases, \eqref{eq_def_weak_derivative} refers to the one-sided derivative from the right or the left, respectively.
	
	Using the terminology from Definition \ref{def_weak_derivative}, Rademacher's result \cite[Lem.~5]{rad16} reads as follows. We recall from Section \ref{sect_introduction} that $K \in \Knn$ is isotropic if $c_K=0$ and $A_K=I_n$.
	
	\begin{proposition}[Rademacher] \label{prop_rademacher} Let $[-1,1] \rightarrow \Knn$, $t \mapsto K_t$ be weakly differentiable at $t=0$. %
	If $K_0$ is isotropic, then $t \mapsto L_{K_t}^{2n}$ is differentiable at $t=0$ with
		\begin{equation}
			\left.\ddt L_{K_t}^{2n}\right|_{t=0}=\frac{1}{\vol(K_0)^3}\left.\ddt\int_{K_t}\norm{x}_2^2 -(n+2)\dint x\right|_{t=0}.
		\end{equation}
	\end{proposition}
	
	To make this article more self-contained, we reproduce the proof of  \cite[Lem.~5]{rad16} with some minor modifications to suit our setting.
	
	\begin{proof}[Proof of Proposition \ref{prop_rademacher}] Recalling \eqref{eq_def_L_K} and using that $t \mapsto K_t$ is weakly differentiable at $t=0$, we have
		\begin{equation} \label{eq_ddt_L_K}
			\left.\ddt L_{K_t}^{2n}\right|_{t=0}=\frac{-2 \det A_{K_t}}{\vol (K_0)^3}\left.\ddt\vol K_t\right|_{t=0}+\frac{1}{\vol(K_0)^2}\left.\ddt\det A_{K_t}\right|_{t=0}.
		\end{equation}
		By definition, we have
		\begin{align}
			A_{K_t} &= \frac1{\vol(K_t)}\int_{K_t}(x-c_{K_t})(x-c_{K_t})^\transp \dint x =\frac1{\vol(K_t)}\int_{K_t}xx^\transp \dint x-c_{K_t}c_{K_t}^\transp
		\end{align}
		for all $t \in [-1,1]$ and, moreover,
		\begin{equation}
			\left.\ddt c_{K_t} c_{K_t}^\transp\right|_{t=0} = \left.\left[ \left(\ddt c_{K_t}\right)_i \left(c_{K_t}\right)_j+\left(c_{K_t}\right)_i \left(\ddt c_{K_t}\right)_j\right]_{i,j=1}^n\right|_{t=0}.
		\end{equation}
		Because $K_0$ is isotropic, we have $c_{K_0}=0$ and hence $\left.\ddt c_{K_t} c_{K_t}^\transp\right|_{t=0}= 0$. This implies
		\begin{equation} \label{eq_ddt_A_K_t}
			\left.\ddt A_{K_t}\right|_{t=0}= \left.\left(-\frac{\ddt \vol(K_t)}{\vol(K_t)^2}\int_{K_t}xx^\transp \dint x+\frac{1}{\vol(K_t)}\ddt \int_{K_{t}}xx^\transp \dint x\right)\right|_{t=0}.
		\end{equation}
		Now Jacobi's formula asserts that
		\begin{equation} \label{eq_jacobis_formula}
			\left.\ddt\det A_{K_t}\right|_{t=0} = \left.\SCPR{\det(A_{K_t})A_{K_t}^{-\transp},\ddt A_{K_t}}\right|_{t=0},
		\end{equation}
		where $\scpr{\cdot,\cdot}$ denotes the Frobenius scalar product. Since $K_0$ is isotropic, we have $(\det A_{K_0})\cdot A_{K_0}^{-\transp}=I_n$ and hence
		\begin{align} \label{eq_ddt_det_A_K_t}
			\left.\ddt\det A_{K_t}\right|_{t=0}&=\left. \left(-\frac{\ddt \vol(K_t)}{\vol(K_t)^2}\int_{K_t}\norm{x}_2^2 \dint x+\frac{1}{\vol(K_t)}\ddt \int_{K_{t}}\norm{x}_2^2 \dint x \right)\right|_{t=0}\\&=\left.\frac{1}{\vol(K_0)}\left(-n\ddt \vol(K_t)+\ddt\int_{K_t}\norm{x}_2^2 \dint x\right)\right|_{t=0}.
		\end{align}
		Plugging this into \eqref{eq_ddt_L_K}, we get
		\begin{equation} 
			\left.\ddt L_{K_t}^{2n}\right|_{t=0}=\frac{1}{\vol (K_0)^3}\left.\left(-2\ddt\vol (K_t)-n\ddt \vol(K_t)+\ddt\int_{K_t}\norm{x}_2^2 \dint x\right)\right|_{t=0}. \qedhere
		\end{equation}
	\end{proof}
	
	Proposition \ref{prop_rademacher} implies the following first-order condition for local maximizers of the isotropic constant.
	
	\begin{corollary} \label{cor_first_order_condition}
		Let $K \in \Knn$ be an isotropic local maximizer of $K \mapsto L_K$. If there exists a weakly differentiable map $[0,1) \rightarrow \Knn$, $t \mapsto K_t$ with $K_{t^*}=K$ for some $t^* \in [0,1)$, then
		\begin{equation} \label{eq_cor_first_order_condition_one_sided} 
			\left.\ddt \int_{K_t} [\norm{x}_2^2 -(n+2)] \dint x \right|_{t=t^*} \leq 0.
		\end{equation}
		If $t^* \in (0,1)$, then we have
		\begin{equation} \label{eq_cor_first_order_condition}
			\left.\ddt \int_{K_t} [\norm{x}_2^2 -(n+2)] \dint x \right|_{t=t^*} = 0.
		\end{equation}
	\end{corollary}
	
	Looking at Corollary \ref{cor_first_order_condition}, it is natural to ask which  derivatives ``can appear'' on the left-hand sides of \eqref{eq_cor_first_order_condition_one_sided} and \eqref{eq_cor_first_order_condition}. In other words, which 
	signed measures can be realized as weak-derivatives at $t=t^*$ of maps $[0,1)\rightarrow \Knn$ with $K_{t^*}=K$? Continuing earlier work by Weisshaupt \cite{wei01}, this question was studied in \cite{kip24}; a characterization of the space of realizable signed measures was given in the special case where $K$ is a polytope.
	
	We now turn to the second derivative of $K \mapsto L_K$. In order to state the result of our computation, we need a notion of a ``sufficiently smooth'' convex-body-valued map. This leads to us to the following definition.
	\begin{definition}
		Let $t_1<0<t_2$ and $K \in \Kno$. A map $[t_1,t_2] \rightarrow \Kno$, $t \mapsto K_t$ with $K_0=K$ is called an \emph{admissible variation of $K$} if
		\begin{enumerate}[label=(\roman*)]
			\item $t \mapsto K_t$ is weakly differentiable at $t=0$, and
			\item the functions $[t_1,t_2] \rightarrow \R$ given by
			\begin{equation}
				t \mapsto\int_{K_t} \norm{x}_2^2 \dint x \quad \text{and} \quad  t\mapsto \vol (K_t)
			\end{equation}
			are twice differentiable at $t=0$.
		\end{enumerate}
	\end{definition}
	In this article, we will study two sorts of admissible variations: RS-movements and linear interpolants of gauge functions. We are now ready to state the main result of this section.
	
	\begin{proposition} \label{prop_L_K_second_derivative}
		Let $[t_1,t_2] \rightarrow \Knn$, $t \mapsto K_t$ be an admissible variation of an isotropic convex body $K_{0}=K \subset \R^n$. If the weak derivative of $t \mapsto K_t$ at $t=0$ is centered, i.e., if
		\begin{equation} \label{eq_prop_L_K_second_derivative_assumption_centered}
			\left.\ddt \int_{K_t} x \dint x \right|_{t=0}= 0,
		\end{equation}
		then $t\mapsto L_{K_t}^{2n}$ is twice differentiable at $t=0$ with
		\begin{align*}
			\left.\ddtn{2}L_{K_t}^{2n}\right|_{t=0}
			&= \frac{1}{\vol(K_t)^4}\left[(n^2+5n+6)\left(\ddt \vol(K_t)\right)^2+\left(\ddt \int_{K_t} \norm{x}_2^2 \dint x\right)^2\right.\\&\quad\left.\left.-(2n+4)\left(\ddt \vol (K_t)\right)\left( \ddt \int_{K_t}\norm{x}_2^2\dint x\right)-\sum_{i,j=1}^n\left(\ddt\int_{K_t}x_ix_j\dint x\right)^2\right]\right|_{t=0}\\
			&\quad+\left.\frac{1}{\vol(K_t)^3} \ddtn{2}\int_{K_t}[\norm{x}_2^2-(n+2)]\dint x \right|_{t=0}.
		\end{align*}
	\end{proposition}
	
	\begin{proof}
		Differentiating the right-hand side of \eqref{eq_ddt_L_K} once again, we have
		\begin{align} \label{eq_ddt2_L_K}
			\quad\quad\ddtn{2}L_{K_t}^{2n}
			&=-\frac{2 \left(\ddt\det A_{K_t}\right)}{\vol(K_t)^3}\ddt\vol (K_t) +\frac{6 \det A_{K_t}}{\vol(K_t)^4}\left(\ddt\vol (K_t)\right)^2\\
			&\quad-\frac{2 \det A_{K_t}}{\vol(K_t)^3}\ddtn{2}\vol (K_t)-\frac{2\ddt \vol (K_t)}{\vol(K_t)^3}\ddt\det A_{K_t}+\frac{1}{\vol(K_t)^2}\ddtn{2}\det A_{K_t}
		\end{align}
		for all $t\in [t_1,t_2]$ for which the second derivatives on the right-hand side exist.
		At first sight, the right-hand side looks complicated, but most of the terms already appeared in the proof of Proposition \ref{prop_rademacher}. The crucial term is $\ddtn{2}\det A_{K_t}$. Applying \eqref{eq_jacobis_formula} twice, we see that this term satisfies
		\begin{align*}
			&\ddtn{2}\det A_{K_t}=\ddt\SCPR{\det(A_{K_t})A_{K_t}^{-\transp}, \ddt A_{K_t}}\\
		&\quad=\SCPR{\ddt\left(\det(A_{K_t})A_{K_t}^{-\transp}\right), \ddt A_{K_t}}+\SCPR{\det(A_{K_t})A_{K_t}^{-\transp}, \ddtn{2} A_{K_t}}\\
			&\quad=\det(A_{K_t})\left[\SCPR{\ddt\left(A_{K_t}^{-\transp}\right), \ddt A_{K_t}}+\SCPR{A_{K_t}^{-\transp},\ddt A_{K_t}}^2+\right.
			\\&\quad\quad\left.\SCPR{A_{K_t}^{-\transp}, \ddtn{2} A_{K_t}}\right]
		\end{align*}
		for all $t\in [t_1,t_2]$ for which the second derivatives on the right-hand side exist.
		
		Using the identity $\ddt A_{K_t}^{-1}=-A_{K_t}^{-1} \cdot \ddt A_{K_t}\cdot A_{K_t}^{-1}$ and the fact that $K$ is isotropic, this implies
		\begin{align} \label{eq_thm_second_derivative_L_K_det_AKt}
			\left.\ddtn{2}\det A_{K_t}\right|_{t=0}=\left.\left(-\NORM{\ddt A_{K_t}}_F^2+\tr\left(\ddt A_{K_t}\right)^2+\tr\left(\ddtn{2} A_{K_t}\right)\right)\right|_{t=0},
		\end{align}
		where $\dotnorm_F$ denotes the Frobenius norm. We now compute each of the three summands on the right-hand side of \eqref{eq_thm_second_derivative_L_K_det_AKt} separately.
		
		Recalling \eqref{eq_ddt_A_K_t}, we have
		\begin{align} \label{eq_thm_second_derivative_L_K_first_term}
			\quad\quad&\left.-\NORM{\ddt A_{K_t}}_F^2\right|_{t=0}=-\frac{1}{\vol(K_0)^2}\left[\NORM{\ddt \vol (K_t) I_n}_F^2\right.\\&\quad\quad\quad\quad\quad\quad\quad\quad\quad\left.\left.-2 \SCPR{\ddt \vol (K_t) I_n, \ddt \int_{K_t}xx^\transp\dint x}+ \NORM{\ddt\int_{K_t}xx^\transp\dint x}_F^2\right]\right|_{t=0}\\
			&\quad\quad\quad=\frac{1}{\vol(K_0)^2}\left[-n\left(\ddt \vol(K_t)\right)^2+2\left(\ddt \vol(K_t)\right)\left( \ddt \int_{K_t}\norm{x}_2^2\dint x\right)\right.\\&\quad\quad\quad\quad\quad\quad\left.\left.-\sum_{i,j=1}^n\left(\ddt\int_{K_t}x_ix_j\dint x\right)^2\right]\right|_{t=0}.
		\end{align}
		Again using \eqref{eq_ddt_A_K_t} in combination with the fact that $K_0$ is isotropic, the second summand becomes
		\begin{align} \label{eq_thm_second_derivative_L_K_second_term}
			\quad\quad\tr\left(\ddt A_{K_t}\right)^2 &= \frac{1}{\vol(K_0)^2}\left[ n^2 \left(\ddt \vol(K_t)\right)^2 \right.\\
			&\quad\quad\quad\left.\left.-2n \ddt \vol(K_t) \ddt \int_{K_t}\norm{x}_2^2\dint x + \left(\ddt \int_{K_t}\norm{x}_2^2\dint x\right)^2  \right]\right|_{t=0}.
		\end{align}
		It remains to compute the third summand. By assumption \eqref{eq_prop_L_K_second_derivative_assumption_centered}, we have
		\begin{equation}
			\left.\ddt c_{K_t}\right|_{t=0} = \left.\frac{-\ddt \vol (K_t)}{\vol(K_t)^2} \int_{K_t} x \dint x + \frac{1}{\vol(K_t)} \ddt \int_{K_t} x \dint x \right|_{t=0} = 0
		\end{equation}
		and hence, again using that $K$ is isotropic,
		\begin{align}
			\left.\ddtn{2} c_{K_t} c_{K_t}^\transp\right|_{t=0} &= \left.\left[ \left(\ddtn{2}c_{K_t}\right)_i \left(c_{K_t}\right)_j+2\left(\ddt c_{K_t}\right)_i \left(\ddt c_{K_t}\right)_j+\left(c_{K_t}\right)_i \left(\ddtn{2}c_{K_t}\right)_j\right]_{i,j=1}^n\right|_{t=0}\nonumber\\&=0.
		\end{align}
		This implies that the third summand satisfies
		\begin{align} \label{eq_thm_second_derivative_L_K_third_term}
			&\left.\tr\left(\ddtn{2} A_{K_t}\right)\right|_{t=0} = \left. \ddtn{2}\E_{X\sim K_t}[\norm{X}_2^2]\right|_{t=0}\\
			&\quad\quad\quad=\left[\left(\frac{2\left(\ddt\vol(K_t)\right)^2}{\vol (K_t)^3}-\frac{\ddtn{2}\vol(K_t)}{\vol(K_t)^2}\right)\int_{K_t}\norm{x}_2^2\dint x\right.\\&\quad\quad\quad\quad\quad\left.\left.-2\frac{\ddt \vol(K_t)}{\vol (K_t)^2}\ddt\int_{K_t}\norm{x}_2^2\dint x+\frac{1}{\vol(K_t)}\ddtn{2} \int_{K_{t}}\norm{x}_2^2 \dint x\right]\right|_{t=0}\\
			&\quad\quad\quad=\left[n\left(\frac{2\left(\ddt\vol(K_t)\right)^2}{\vol (K_t)^2}-\frac{\ddtn{2}\vol(K_t)}{\vol(K_t)}\right)\right.\\&\quad\quad\quad\quad\quad\left.\left.-2\frac{\ddt \vol(K_t)}{\vol (K_t)^2}\ddt\int_{K_t}\norm{x}_2^2\dint x+\frac{1}{\vol(K_t)}\ddtn{2} \int_{K_{t}}\norm{x}_2^2 \dint x\right]\right|_{t=0}.
		\end{align}
		Plugging our expressions \eqref{eq_thm_second_derivative_L_K_first_term}, \eqref{eq_thm_second_derivative_L_K_second_term} and \eqref{eq_thm_second_derivative_L_K_third_term} for the three summands into the original formula \eqref{eq_thm_second_derivative_L_K_det_AKt}, we obtain
		\begin{align} \label{eq_thm_second_derivative_L_K_eq_3}
			\quad\quad&\left.\ddtn{2}\det A_{K_t}\right|_{t=0} = \frac{1}{\vol (K_t)^2}\left[(n^2+n)\left(\ddt \vol(K_t)\right)^2+\left(\ddt \int_{K_t} \norm{x}_2^2 \dint x\right)^2\right.\\&\quad\quad\quad\quad\quad\left.\left.-2n\left(\ddt \vol (K_t)\right)\left( \ddt \int_{K_t}\norm{x}_2^2\dint x\right)-\sum_{i,j=1}^n\left(\ddt\int_{K_t}x_ix_j\dint x\right)^2\right]\right|_{t=0}\\
			&\quad\quad\quad\quad\quad\left.+\frac{1}{\vol(K_t)} \ddtn{2}\int_{K_t}[\norm{x}_2^2-n]\dint x\right|_{t=0}.
		\end{align}
		Finally, we observe that \eqref{eq_ddt2_L_K}, \eqref{eq_ddt_det_A_K_t} and the fact that $K$ is isotropic imply
		\begin{align} \label{eq_thm_second_derivative_L_K_eq_4}
			\left.\ddtn{2}L_{K_t}^{2n}\right|_{t=0}
			&=\left[\frac{-4 \left(\ddt\det A_{K_t}\right)}{\vol (K_t)^3}\ddt\vol K_t +\frac{6}{\vol (K_t)^4}\left(\ddt\vol K_t\right)^2\right.\\
			&\quad\quad\quad\left.\left.-\frac{2}{\vol (K_t)^3}\ddtn{2}\vol K_t+\frac{1}{\vol(K_t)^2}\ddtn{2}\det A_{K_t}\right]\right|_{t=0}\\
			&=\left[\frac{-4 \left(\ddt\int_{K_t}\norm{x}_2^2 \dint x\right)}{\vol( K_t)^4}\ddt\vol K_t +\frac{4n+6}{\vol (K_t)^4}\left(\ddt\vol K_t\right)^2\right.\\
			&\quad\quad\quad\left.\left.-\frac{2}{\vol (K_t)^3}\ddtn{2}\vol K_t+\frac{1}{\vol(K_t)^2}\ddtn{2}\det A_{K_t}\right]\right|_{t=0}.
		\end{align}
		The claim follows by plugging \eqref{eq_thm_second_derivative_L_K_eq_3} into \eqref{eq_thm_second_derivative_L_K_eq_4}.
	\end{proof}
	
	\begin{remark}
		Let $[t_1,t_2] \rightarrow \Knn$, $t \mapsto K_t$ be as in Proposition \ref{prop_L_K_second_derivative}. If we are interested in a second-order necessary condition for local maximizers of the isotropic constant, then it makes sense to assume that the first derivative vanishes, as the second-order condition is redundant otherwise. If we have $\left.\ddt L_{K_t}^{2n}\right|_{t=0} = 0$, then Proposition \ref{prop_rademacher} implies $\ddt \left.\int_{K_t} \norm{x}_2^2 \dint x \right|_{t=0} = (n+2) \ddt \left. \vol(K_t)^2 \right|_{t=0}$ and the rather complicated expression from Proposition \ref{prop_L_K_second_derivative} simplifies to
		\begin{align*}
			\left.\ddtn{2}L_{K_t}^{2n}\right|_{t=0}
			&=\left[\frac{n+2}{(\vol K_0)^4}\left(\ddt\vol K_t\right)^2 -\frac{1}{\vol(K_0)^4}\sum_{i,j=1}^n\left(\ddt\int_{K_t}x_ix_j\dint x\right)^2\right.\\
			&\left.\left.\quad\quad+\frac{1}{\vol(K_0)^3}\left(\ddtn{2} \int_{K_{t}}[\norm{x}_2^2 -(n+2)]\dint x\right)\right]\right|_{t=0}.
		\end{align*}
	\end{remark}

	\section{Shadow systems and generalized RS-movements} \label{sect_shadow_systems}
	Let $K \in \Knn$ and $u \in \sphere$. A \emph{shadow system} is a one-parameter family of convex bodies of the form
	\begin{equation}
		K_t = \conv \{x +t\beta(x)  u \mid x \in K\} \quad \text{for } t \in [t_1,t_2],
	\end{equation}
	where $\beta \colon K \rightarrow \R$ is an arbitrary function. Motivated by certain isoperimetric-type inequalities, shadow systems were first studied by Rogers and Shephard \cite{rs58}, who showed that the volume $\vol(K_t)$ of a shadow system is a convex function in $t$. Shadow systems have also turned out to be a useful tool in the study of the Mahler volume \cite{cg06,mr06,fmz12,sar13,afz19,mr19}.
	
	In the introduction, we have already encountered a special class of shadow systems, namely, RS-movements. An RS-movement can be defined as a shadow system with the additional requirements that 
	\begin{enumerate}[label=(\roman*)]
		\item \label{item_rs_movement_1} the \emph{speed function} $\beta$ is constant on each chord $K \cap [x+\linh \{u\}]$ for $x \in K$, and
		\item \label{item_rs_movement_2} the set $\{x +t\cdot\beta(x)u \mid x \in K\}$ is already convex for all $t \in [t_1,t_2]$.
	\end{enumerate}
	We note that the speed function $\beta$ of an RS-movement can be affine (on $K$). In this case, $K_t$ is an affine image of $K$ for all $t \in [t_1,t_2]$. Because such RS-movements provide no information about extremizers of affinely invariant functionals, the definition of RS-decomposability (Definition \ref{def_RS_decomposable}) requires $\beta$ to be non-affine. %

	RS-movements were introduced by Campi, Colesanti and Gronchi \cite{ccg99} to study local maximizers of Sylvester-type functionals. Theorem \ref{thm_ccg} is a special case of \cite[Thm.~3.1]{ccg99}, corresponding to the exponent $r=2$. The latter result is stated in the setting of Sylvester's problem and addresses arbitrary exponents. For the connection between Sylvester's problem and isotropic constants, we refer to \cite{kin69}.
	
	In this section, we study a class of shadow systems that is situated between RS-movements and arbitrary shadow systems. Relaxing condition \ref{item_rs_movement_1}, we allow that the speed function $\beta$ is affine (rather than constant) on each chord $K \cap [x+\linh \{u\}]$ that intersects the hyperplane $u^\bot$. The purpose of this somewhat technical condition will become clear below.	
	\begin{definition} \label{def_generalized_RS_movement}
		Let $K \in \Knn$ be centered, $u \in \sphere$ and $\beta \colon K \rightarrow \R$. Moreover, let $t_1< 0 < t_2 \in \R$ and $L \coloneqq \linh \{u\}$. We call the shadow system
		\begin{equation}
			[t_1,t_2] \rightarrow \Knn, \quad t \mapsto K_t \coloneqq \conv \{x +t\cdot\beta(x)u \mid x \in K\}
		\end{equation}
		a \emph{generalized RS-movement} of $K$ if
		\begin{enumerate}[label=(\roman*)]
			\item \label{def_generalized_RS_movement_1} $\beta$ is affine on $(x+L) \cap K$ for all $x \in K$,
			\item \label{def_generalized_RS_movement_2} $\beta$ is constant on $(x+L) \cap K$ for all $x\in \pi_{u^\bot}(K)$ with $x\notin K$, and
			\item \label{def_generalized_RS_movement_3} the set $\{x +t\cdot\beta(x)u \mid x \in K\}$ is a full-dimensional convex body for all $t \in [t_1,t_2]$.
		\end{enumerate}
		Given a generalized RS-movement with speed function $\beta$, we define two functions $\pi_{u^\bot}(K) \rightarrow \R$ via
		\begin{equation} \label{eq_def_beta_plus_minus}
			\beta^+(x) \coloneqq \beta\left(\argmax_{y \in (x+L) \cap K} \scpr{y,u}\right) \quad \text{and} \quad \beta^-(x) \coloneqq \beta\left(\argmin_{y \in (x+L) \cap K} \scpr{y,u}\right).
		\end{equation}
	\end{definition}
	
	For an illustration of Definition \ref{def_generalized_RS_movement}, we refer to Figure \ref{fig_generalized_RS_movement}.  We start with some general properties of generalized RS-movements.
	\begin{lemma} \label{lemma_generalized_RS_movement_basic_properties}
		Let $[t_1,t_2] \rightarrow \Knn$, $t \mapsto K_t$ be a generalized RS-movement in direction $u$ with speed function $\beta$. Then
		\begin{enumerate}[label=(\roman*)]
			\item \label{lemma_generalized_RS_movement_basic_properties_1} for all $x \in \relint \pi_{u^\bot}(K)$ and $t\in [t_1,t_2]$, we have
			\begin{equation} \label{eq_lemma_generalized_RS_movement_basic_properties}
				f_{K_t,u}(x) = f_{K_0,u}(x) + t \cdot \beta^-(x) \quad \text{and} \quad g_{K_t,u}(x) = g_{K_0,u}(x) - t \cdot \beta^+(x),
			\end{equation}
		where $f_{K_t,u}$ and $g_{K_t,u}$ are defined as \eqref{eq_def_f_and_g},
		\item \label{lemma_generalized_RS_movement_basic_properties_2} $\beta^+$ and $\beta^-$ are continuous on $\relint \pi_{u^\bot}(K)$, and
		\item \label{lemma_generalized_RS_movement_basic_properties_3} if $\beta$ does not vanish, then there exists a point $x \in \relint \pi_{u^\bot}(K)$ with $\beta^+(x) \neq 0$ or $\beta^-(x) \neq 0$.
		\end{enumerate}
	\end{lemma}
	
	\begin{figure}[t]
		\begin{subfigure}[c]{.4\linewidth}
			\centering
			\input{tikz/1/beta_and_K.tex}
			
			\caption{$K$ and $\beta$ \label{subfig_P_and_f}}
		\end{subfigure}
		\begin{subfigure}[c]{0.4\linewidth}
			\centering
			\input{tikz/1/K01.tex}
			\vspace*{0.5cm}
			
			\input{tikz/1/K02.tex}
			
			\caption{two elements of $(K_t)_{t \in [0,1]}$}%
	\end{subfigure}
	\caption{\label{fig_generalized_RS_movement} An illustration of Definition \ref{def_generalized_RS_movement} in the case where $K$ is a polygon. In Subfigure (A), the red area corresponds to the set where $\beta^+$ and $\beta^-$ have to be equal by condition \ref{def_generalized_RS_movement_2}. In Subfigure (B), the sets $K_{0.1} \setminus K$ and $K_{0.2} \setminus K$ are shaded in cyan, whereas the sets $K \setminus K_{0.1}$ and $K \setminus K_{0.2}$ are shaded in orange.}
	\end{figure}
	
	\begin{proof}
		We begin with \ref{lemma_generalized_RS_movement_basic_properties_1}. Let $x \in \relint \pi_{u^\bot}(K)$ and let
		\begin{equation}
			v \coloneqq \argmax_{y \in (x+L) \cap K} \scpr{y,u} \quad \text{and} \quad w \coloneqq \argmin_{y \in (x+L) \cap K} \scpr{y,u}.
		\end{equation}
		By construction, the points $v$ and $w$ move with constant velocities $\beta^+(x)$ and $\beta^-(x)$, respectively, along the line $x+L$, as $t$ varies from $t_1$ to $t_2$. We claim that the chord $(x+L)\cap K_t$ is the convex hull of $v+t \cdot \beta(v)$ and $w+t \cdot \beta(w)$ throughout the generalized RS-movement, i.e., that we have
		\begin{equation}
			\{y+t\cdot \beta(y) u \mid y \in (x+L)\cap K\} = [v+t \cdot \beta(v), w+t \cdot \beta(w)]= [v+t \cdot \beta^+(x), w+t \cdot \beta^-(x)]
		\end{equation}
		for all $t \in [t_1,t_2]$. To see this, we observe that, since $\beta$ is affine on each chord,
		\begin{align}
			\lambda v + (1-\lambda) w + t\cdot\beta(\lambda v + (1-\lambda) w) &= \lambda v + (1-\lambda) w + \lambda \cdot t\cdot\beta(v) + (1-\lambda)\cdot t\cdot \beta(w)\\
			&= \lambda (v+t\cdot\beta(v)) + (1-\lambda) (w +t\cdot \beta(w))
		\end{align}
		for all $\lambda \in [0,1]$ and $t \in [t_1,t_2]$. To prove \eqref{eq_lemma_generalized_RS_movement_basic_properties}, we have to show that the moving points $v$ and $w$ do not pass each other, i.e., that we have
		\begin{equation}
			\scpr{w+t\cdot \beta^-(x),u} < \scpr{v+t\cdot \beta^+(x),u} 
		\end{equation}
		for all $t \in [t_1,t_2]$. We assume towards a contradiction that there exists a $t^* \in [t_1,t_2]$ with $w+t^*\cdot \beta^-(x)=v+t^*\cdot \beta^+(x)$. Then $\volt{1}[(x+L)\cap K_{t^*}] =0$, and since
		\begin{equation}
			h\colon \pi_{u^\bot}(K) \rightarrow \R, \quad y \mapsto \volt{1}[(y+L)\cap K_{t^*}]
		\end{equation}
		is concave and non-negative, it follows that $h$ is identically zero on $\pi_{u^\bot}(K)$. This is a contradiction to the assumption that $K_{t^*}$ is a full-dimensional convex body.
		
		To prove \ref{lemma_generalized_RS_movement_basic_properties_2}, we note that by \ref{lemma_generalized_RS_movement_basic_properties_1} there exists a $t \neq 0$ with
		\begin{equation}
			\beta^-|_{\relint \pi_{u^\bot}(K)} = \frac{1}{t} \left(f_{K_t,u} - f_{K_0,u}\right)|_{\relint \pi_{u^\bot}(K)}
		\end{equation}
		and the right-hand side is continuous because the convex functions $f_{K_t,u}$ and $f_{K_0,u}$ are continuous on $\relint \pi_{u^\bot}(K)$. The continuity of $\beta^+|_{\relint \pi_{u^\bot}(K)}$ follows similarly.
		
		For \ref{lemma_generalized_RS_movement_basic_properties_3}, we first observe that
		\begin{equation}
			K_t = \cl [K_t \cap \pi_{u^\bot}^{-1}(\relint \pi_{u^\bot}(K))],
		\end{equation}
		which follows from the fact that $K_t$ is closed and convex for all $t\in[t_1,t_2]$. If there exists a point $x \in K$ with $\beta(x) \neq 0$, then there exists a $t\in [t_1,t_2]$ with $K_t \neq K_0$ and hence
		\begin{equation}
			K_t \cap \pi_{u^\bot}^{-1}(\relint \pi_{u^\bot}(K)) \neq K_0 \cap \pi_{u^\bot}^{-1}(\relint \pi_{u^\bot}(K)).
		\end{equation}
		Now the claim follows from \ref{lemma_generalized_RS_movement_basic_properties_1}.
	\end{proof}

	We have already mentioned in Section \ref{sect_second_derivative} that RS-movements are an example of admissible variations. The following lemma shows that the same can be said about generalized RS-movements.
	\begin{lemma} \label{lemma_generalized_RS_movement_admissible_variation}
		Let $[t_1,t_2]\rightarrow \Knn, t \mapsto K_t$ be a generalized RS-movement of $K\in \Kno$ in direction $u$ with speed function $\beta$. Then $[t_1,t_2]\rightarrow \Knn, t \mapsto K_t$ is an admissible variation of $K$ with
		\begin{equation}
			\ddt \int_{K_t} h(x) \dint x = \int_{\pi_{u^\bot}(K)} \left(\beta^+(x) \cdot h(x-g_{K_t,u}(x)\cdot u)- \beta^-(x) \cdot h(x+f_{K_t,u}(x)\cdot u)\right) \dint x
		\end{equation}
		for all $h \in C(\R^n)$, $t \in [t_1,t_2]$ and
		\begin{align}
			\left.\ddtn{2}\int_{K_t} [\norm{x}_2^2-(n+2)] \dint x\right|_{t=0} =-2\int_{\pi_{u^\bot}(K)} \left(\beta^+(x)^2 \cdot g_{K,u}(x) +\beta^-(x)^2 \cdot f_{K,u}(x)\right) \dint x.
		\end{align}
	\end{lemma}
	
	\begin{proof}
		Let $h \in C(\R^n)$. By Fubini's theorem, we have
		\begin{equation}
			\int_{K_t} h(x) \dint x = \int_{\pi_{u^\bot}(K_t)} \int_{f_{K_t,u}(x)}^{-g_{K_t,u}(x)} h(x+ s \cdot u) \dint s \dint x.
		\end{equation}
		Using that $\volt{n-1}[\relbd \pi_{u^\bot}(K)]=0$ and applying the Leibniz integral rule \cite[Thm.~2.27]{fol99} in combination with Lemma \ref{lemma_generalized_RS_movement_basic_properties}\ref{lemma_generalized_RS_movement_basic_properties_1}, we get
		\begin{align}
			\ddt\int_{K_t} h(x) \dint x &= \int_{\relint \pi_{u^\bot}(K)} \ddt\int_{f_{K_t,u}(x)}^{-g_{K_t,u}(x)} h(x+s \cdot u) \dint s \dint x\\
			=&\int_{\pi_{u^\bot}(K)} \left(\beta^+(x) \cdot h(x-g_{K_t,u}(x)\cdot u)- \beta^-(x) \cdot h(x+f_{K_t,u}(x)\cdot u)\right) \dint x.
		\end{align}
		Applying the Leibniz integral rule a second time, we obtain
		\begin{align}
			\left.\ddtn{2}\int_{K_t} \norm{x}_2^2 \dint x\right|_{t=0} =-2\int_{\pi_{u^\bot}(K)} \left(\beta^+(x)^2 \cdot g_{K,u}(x) +\beta^-(x)^2 \cdot f_{K,u}(x)\right) \dint x
		\end{align}
		and
		\begin{equation}
			\left.\ddtn{2} \vol(K_t) \right|_{t=0} =0. \qedhere
		\end{equation}
	\end{proof}
	
	In Definition \ref{def_generalized_RS_movement}, we required that the speed function $\beta$ is constant on all chords that do not intersect the hyperplane $u^\bot$. The purpose of this somewhat technical condition is to ensure the validity of the following lemma.
	
	\begin{lemma} \label{lemma_generalized_RS_movement_second_derivative}
		Let $K \in \Kno$ and let $[t_1,t_2]\rightarrow \Knn, t \mapsto K_t$ be a generalized RS-movement of $K$ with non-vanishing speed function $\beta$. Then we have
		\begin{equation} \label{eq_lemma_generalized_RS_movement_second_derivative}
			\left.\ddtn{2}\int_{K_t} [\norm{x}_2^2-(n+2)] \dint x \right|_{t=0}>0.
		\end{equation}
	\end{lemma}
	\begin{proof}
		We already know from Lemma \ref{lemma_generalized_RS_movement_admissible_variation} that
		\begin{align}
			\left.\ddtn{2}\int_{K_t} [\norm{x}_2^2-(n+2)] \dint x\right|_{t=0} =-2\int_{\pi_{u^\bot}(K)} \left(\beta^+(x)^2 \cdot g_{K,u}(x) +\beta^-(x)^2 \cdot f_{K,u}(x)\right) \dint x.
		\end{align}
		Let $x \in \pi_{u^\bot}(K)$. We claim that
		\begin{equation}\label{eq_lemma_generalized_RS_movement_claim}
			\beta^+(x)^2 \cdot g_{K,u}(x) +\beta^-(x)^2 \cdot f_{K,u}(x) \leq 0. 
		\end{equation}
		To see this, we distinguish two cases:
		\begin{enumerate}[label=(\alph*)]
			\item If $x \in K$, then $f_{K,u}(x) \leq 0 \leq -g_{K,u}(x)$ and \eqref{eq_lemma_generalized_RS_movement_claim} follows.
			\item If $x \notin K$, then condition \ref{def_generalized_RS_movement_2} in Definition \ref{def_generalized_RS_movement} implies that $\beta^+(x)=\beta^-(x)$. In combination with $f_{K,u}(x) \leq -g_{K,u}(x)$,  \eqref{eq_lemma_generalized_RS_movement_claim} follows.
		\end{enumerate}
		By Lemma \ref{lemma_generalized_RS_movement_basic_properties}\ref{lemma_generalized_RS_movement_basic_properties_2}, the left-hand side of \eqref{eq_lemma_generalized_RS_movement_claim} is a continuous function of $x$ on $\relint \pi_{u^\bot}(K)$. Therefore, it remains to show that there exists an $x\in \relint \pi_{u^\bot}(K)$ where $\eqref{eq_lemma_generalized_RS_movement_claim}$ holds with strict inequality.
		By Lemma \ref{lemma_generalized_RS_movement_basic_properties}\ref{lemma_generalized_RS_movement_basic_properties_3}, there exists a point $x \in \relint\pi_{u^\bot}(K)$ with $\beta^+(x)\neq 0$ or $\beta^-(x)\neq 0$. Without loss generality, let $\beta^+(x) \neq 0$. By the continuity of $\beta^+$ on $\relint \pi_{u^\bot}(K)$, there exists a neighborhood $B \subset \pi_{u^\bot}(K)$ of $x$ with $\beta^+(y) \neq 0$ for all $y \in B$. We assume towards a contradiction that
		\begin{equation} \label{eq_lemma_generalized_RS_movement_negation_of_claim}
			\beta^-(x)^2 \cdot f_{K,u}(x) =-\beta^+(x)^2 \cdot g_{K,u}(x) \quad \text{for all } x \in B. 
		\end{equation}
		Because $0 \in \interior K$, the set $N \coloneqq \{x \in \pi_{u^\bot}(K) \mid g_{K,u}(x)=0\}$ has volume zero. On $B\setminus N$, \eqref{eq_lemma_generalized_RS_movement_negation_of_claim} implies that $f_{K,u}(x)$ and $g_{K,u}(x)$ are non-zero and have opposing signs. But now property \ref{def_generalized_RS_movement_2} from Definition \ref{def_generalized_RS_movement} implies that $\beta^+(y)=\beta^-(y)$ for all $y \in B \setminus N$. From \eqref{eq_lemma_generalized_RS_movement_negation_of_claim}, we get that $f_{K,u}(x) = - g_{K,u}(x)$ for all $y \in B \setminus N$, a contradiction to the fact that $K_0$ is full-dimensional. 
	\end{proof}
	
	We now have all prerequisites at hand that are needed for our proof of Theorem \ref{thm_ccg}.
	
	\begin{proof}[Proof of Theorem \ref{thm_ccg}]
		Let $s \in (-1,1)$. Our goal is to show that $\left.\ddtn{2}L_{K_t}^{2n}\right|_{t=s}>0$. Since the isotropic constant is an affine invariant, we can assume without loss of generality that $K_s$ is isotropic and that $u$ is the $n$-th standard basis vector $e_n$. Moreover, by a simple reparameterization of the RS-movement $(K_t)_{t \in [-1,1]}$, we can assume that $s=0$. Let $V$ be the $n$-dimensional vector space of affine functions $\pi_{u^\bot}(K) \rightarrow \R$. Since $\beta$ is the speed function of an RS-movement, there exists a function $\tilde{\beta} \colon \pi_{u^\bot}(K) \rightarrow \R$ with $\beta = \tilde{\beta} \circ \pi_{u^\bot}$. By Definition \ref{def_RS_decomposable}, $\tilde{\beta}$ is not affine, and hence the vector space
		\begin{equation}
			W \coloneqq \linh \{\tilde{\beta}\} \oplus V \subset C(\pi_{u^\bot}(K))
		\end{equation}
		is $(n+1)$-dimensional. Moreover, every function $\alpha \in W$ induces an RS-movement $t \mapsto K_t^\alpha$ with speed function $\alpha \circ \pi_{u^\bot}$. That $K_t^\alpha$ is indeed convex for all $t \in [-1,1]$ follows from the fact that $K_t^\alpha$ is simply an affine image of $K_t$ for all $t \in [-1,1]$. By Lemma \ref{lemma_generalized_RS_movement_admissible_variation}, the map
		\begin{equation}
			A \colon W \mapsto \R^{n-1} \times \R, \quad A(\alpha) \coloneqq \left(\left[\left.\ddt \int_{K^\alpha_t} x_i x_n \dint x\right|_{t=0}\right]_{i=1}^{n-1},\left.\ddt \int_{K^\alpha_t} x_n \dint x\right|_{t=0}\right)
		\end{equation}
		is linear, and since
		\begin{equation}
			\dim (\R^{n-1} \times \R)  < n+1 =\dim W,
		\end{equation}
		the map $A$ has a non-trivial kernel. Let $\alpha \in \ker A \setminus\{0\}$. Because $K_t^\alpha$ is an affine image of $K_t$, we have $L_{K_t^\alpha}=L_{K_t}$ for all $t \in [-1,1]$. To prove the theorem, it remains to show that $\left.\ddtn{2}L_{K_t^\alpha}^{2n}\right|_{t=0}>0$. For this, let $i,j \in [n-1]$. %
		Using Fubini's theorem and the fact that $(K^\alpha_t)_{t\in[-1,1]}$ is an RS-movement, we have
		\begin{align}
			\int_{K^\alpha_t} x_i x_j \dint x &= \int_{\pi_{u^\bot}(K)} [-g_{K_t^\alpha,u}(x)-f_{K_t^\alpha,u}(x)]\cdot x_i x_j \dint x \\&= \int_{\pi_{u^\bot}(K)} [-g_{K,u}(x)-f_{K,u}(x)]\cdot x_i x_j \dint x = \int_{K} x_i x_j \dint x
		\end{align}
		for all $t \in [-1,1]$ and hence $\left.\ddt \int_{K^\alpha_t} x_i x_j \dint x\right|_{t=0} = 0$. By a similar argument, we conclude that $\left.\ddt \int_{K^\alpha_t} x_i \dint x\right|_{t=0} = 0$ for all $i\in [n-1]$. Recalling that $\alpha \in \ker A$, we observe that the map $t \mapsto K_t^\alpha$ satisfies the assumptions of Proposition \ref{prop_L_K_second_derivative}. Because $(K^\alpha_t)_{t\in[-1,1]}$ is an RS-movement, we have $\left.\ddt \vol(K^\alpha_t)\right|_{t=0}=0$, and Proposition \ref{prop_L_K_second_derivative} yields
		\begin{align}
			\left.\ddtn{2}L_{K^\alpha_t}^{2n}\right|_{t=0}
			&= \frac{1}{\vol(K_0)^4}\left.\left[\left(\ddt \int_{K^\alpha_t} \norm{x}_2^2 \dint x\right)^2-\sum_{i,j=1}^n\left(\ddt\int_{K^\alpha_t}x_ix_j\dint x\right)^2\right]\right|_{t=0}\\
			&\quad\quad\quad\left.+\frac{1}{\vol(K_0)^3} \ddtn{2}\int_{K^\alpha_t}[\norm{x}_2^2-(n+2)]\dint x\right|_{t=0}.
		\end{align}
		Since the matrix $\left.\ddt\int_{K^\alpha_t}xx^\transp\dint x\right|_{t=0}$ has only one non-zero entry, which is located in the $n$-th row and the $n$-th column, we have
		\begin{equation}
			\left(\ddt \int_{K^\alpha_t} \norm{x}_2^2 \dint x\right)^2= \left(\ddt \int_{K^\alpha_t} x_n^2 \dint x\right)^2= \sum_{i,j=1}^n\left(\ddt\int_{K^\alpha_t}x_ix_j\dint x\right)^2
		\end{equation}
		and, using Lemma \ref{lemma_generalized_RS_movement_second_derivative}, we conclude that
		\begin{equation}
			\left.\ddtn{2}L_{K^\alpha_t}^{2n}\right|_{t=0}
			=\left.\frac{1}{\vol(K_0)^3} \ddtn{2}\int_{K^\alpha_t}[\norm{x}_2^2-(n+2)]\dint x\right|_{t=0}>0. \qedhere
		\end{equation}
	\end{proof}
	
	As a byproduct of the concepts developed in this section, we obtain the following necessary condition for local maximizers of the isotropic constant.
	
	\begin{proposition} \label{prop_generalized_RS_movements}
		Let $K \in \Knn$ be isotropic and let $t \mapsto K_t$ be a generalized RS-movement of $K$ with non-vanishing speed function. If $t \mapsto K_t$ satisfies
		\begin{equation} \label{eq_prop_generalized_RS_movements}
			\left.\ddt\int_{K_t}x_ix_j\dint x\right|_{t=0} = 0 \quad \text{and} \quad \left.\ddt\int_{K_t}x_i\dint x\right|_{t=0} = 0 \quad \text{for all } i, j \in \{1,\dots,n\},
		\end{equation}
		then $K$ is not a local maximizer of the isotropic constant.
	\end{proposition}
	
	\begin{proof}
		We assume towards a contradiction that $K$ is a local maximizer. By Corollary \ref{cor_first_order_condition}, we have
		\begin{equation}
			(n+2) \left.\ddt \vol(K_t) \right|_{t=0} = \left. \ddt\int_{K_t} \norm{x}_2^2 \dint x \right|_{t=0} = 0.
		\end{equation}
		Now Proposition \ref{prop_L_K_second_derivative} and Lemma \ref{lemma_generalized_RS_movement_second_derivative} imply that
		\begin{equation}
			\left.\ddtn{2}L_{K_t}^{2n}\right|_{t=0}
			=\frac{1}{\vol(K)^3} \left. \ddtn{2}\int_{K_t}[\norm{x}_2^2-(n+2)]\dint x\right|_{t=0}>0,
		\end{equation}
		which is a contradiction to the assumption that $K$ is a local maximizer of $K \mapsto L_K$.
	\end{proof}
	
	Although Proposition \ref{prop_generalized_RS_movements} sheds some light on the geometric mechanism underlying Theorem \ref{thm_ccg}, the condition \eqref{eq_prop_generalized_RS_movements} may seem somewhat contrived. What qualitative features of the convex body $K$ can ensure the existence of a generalized RS-movement of $K$ that fulfills the quantitative condition \eqref{eq_prop_generalized_RS_movements}? One possible answer is suggested by reconsidering the dimension argument that we invoked in the proof of Theorem \ref{thm_ccg}. %
	\begin{definition} \label{def_GRS}
		For $K \in \Kno$ and $u \in \sphere$, let $\GRS(K,u)$ be the set of functions $\beta \colon K \rightarrow \R$ such that $\beta$ is the speed function of a generalized RS-movement $[a^\beta,b^\beta] \rightarrow \Knn$, $t \mapsto K_t^\beta$ of $K$ in direction $u$.
	\end{definition}
	We note that the domain $[a^\beta,b^\beta]$ in Definition \ref{def_GRS} is allowed to depend on $\beta$. This entails that $\GRS(K,u)$ is a vector space. Because $\GRS(K,u)$ contains all affine functions on $\pi_{u^\bot}(K)$, we always have $\dim \GRS(K,u) \geq n$.
	By Lemma \ref{lemma_generalized_RS_movement_admissible_variation}, the map
	\begin{equation}
		a_h \colon \GRS(K,u) \rightarrow \R, \quad a_h(\beta) \coloneqq \left. \ddt\int_{K^\beta_t} h(x) \dint x \right|_{t=0}
	\end{equation}
	is linear, and if $\dim \GRS(K,u)$ is larger than the number of linear equations in \eqref{eq_prop_generalized_RS_movements}, then the existence of a non-trivial generalized RS-movement of $K$ in direction $u$ that satisfies \eqref{eq_prop_generalized_RS_movements} is a matter of elementary linear algebra (for details, see the proof of Theorem \ref{thm_minkowski_summands} below). We record this fact as a corollary.
	
	\begin{corollary} \label{cor_generalized_RS_movements_dim}
		Let $K \subset \R^n$ be an isotropic convex body and let $\GRS(K,u)$ be the vector space of speed functions of generalized RS-movements of $K$ in direction $u$. If
		\begin{equation} \label{eq_cor_generalized_RS_movements_dim}
			\dim \GRS(K,u) > \frac{n(n+1)}{2}+n= \frac{n^2+3n}{2},
		\end{equation}
		then $K$ is not a local maximizer of the isotropic constant.
	\end{corollary}

	Again, one might ask about the geometric ``meaning'' of Corollary \ref{cor_generalized_RS_movements_dim}. In Section \ref{sect_the_polytopal_case}, we will analyze the structure of $\GRS(K,u)$ in the case where $K$ is a polytope. As one might suspect, condition \ref{def_generalized_RS_movement_2} in Definition \ref{def_generalized_RS_movement} is somewhat of a nuisance, not least because it involves an orthogonal projection, where the notion of orthogonality depends on the covariance matrix $A_K$. It is therefore natural to ask whether the condition can be dispensed with. In the following section, we prove Theorem \ref{thm_minkowski_summands}, which is analogous to Corollary \ref{cor_generalized_RS_movements_dim}, except that the generalized RS-movements are replaced by linear interpolants of gauge functions, which allows us to get rid of condition \ref{def_generalized_RS_movement_2} in Definition \ref{def_generalized_RS_movement}. A detailed comparison of Theorem \ref{thm_minkowski_summands} and Corollary \ref{cor_generalized_RS_movements_dim} will be given in Section \ref{sect_the_polytopal_case}.
	
	\begin{remark}
		Building on the notion of degrees of freedom of log-concave functions due to Fradelizi-Guédon \cite{fg06}, a decomposability condition similar to \eqref{eq_cor_generalized_RS_movements_dim} is used by Melbourne-Nayar-Roberto \cite{mnr25} to identify the minimizers of entropy in the class of (one-dimensional) log-concave random variables with a given variance (see also \cite{mnt20} for the symmetric case). As shown by Fradelizi and Marín Sola \cite[Prop.~5.2]{fms24}, the strong isotropic constant conjecture (for all dimensions $n \in \N$) is equivalent to a certain conjecture about the minimizers of entropy among isotropic log-concave random vectors.
	\end{remark}

	\section{Minkowski decomposability of the polar body} \label{sect_minkowski_decomposability}
	
	In this section, we prove our main result, Theorem \ref{thm_minkowski_summands}.
	
	The notion of a gauge function of a convex body $K \in \Kno$ was already introduced in Section \ref{sect_introduction}. Closely related to the gauge function is the \emph{radial function} $\rho(K,\cdot)\colon \sphere \rightarrow \R$ of a convex body $K\in \Kno$, which is given by
	\begin{equation}
		\rho(K,u) \coloneqq \max \{\lambda \geq 0 \mid \lambda u \in K\}=\frac{1}{\norm{u}_K}.
	\end{equation}
	For $f \in \Subl(\R^n)$, we define 
	\begin{equation}
		\Lpm(f) \coloneqq \{g \colon \R^n \rightarrow \R \mid \exists \varepsilon >0 \colon f+ \varepsilon g, f- \varepsilon g \in \Subl(\R^n)\}.
	\end{equation}
	We note that every element of $\Lpm(f)$ is positively homogeneous and, as a scaled difference of two sublinear functions, continuous. Because $\Subl(\R^n)$ is convex, the condition $f+ \varepsilon g, f- \varepsilon g \in \Subl(\R^n)$ ensures that the whole line segment $[f+ \varepsilon g, f- \varepsilon g]$ is contained in $\Subl(\R^n)$. Using this fact, it is easy to show that $\Lpm(f)$ is a vector space. In fact, $\Lpm(f)$ is the lineality space of the support cone of $\Subl(\R^n)$ at $f$. As described in the introduction, the dimension of the vector space $\Lpm(\dotnorm_K)$ is the pivotal quantity for our argument; if it is too large, then $K$ cannot be a local maximizer of the isotropic constant. We note that Theorem \ref{thm_minkowski_summands} does not explicitly mention the quantity $\dim\Lpm(\dotnorm_K)$, but rather $\dim \mathcal{S}(K^\circ)$. The following elementary lemma shows that these quantities coincide.
	
	\begin{lemma} \label{lemma_L_is_span_of_S}
		Let $K \in \Kno$. Let $\iota \colon \Kn \rightarrow \Subl(\R^n)$ be the  bijection that maps each element of $\Kn$ to its support function. Then we have $\linh\iota(\mathcal{S}(K^\circ)) = \Lpm(\dotnorm_K)$
		and hence $\dim \mathcal{S}(K^\circ) = \dim \Lpm(\dotnorm_K)$.
	\end{lemma}
	
	\begin{proof}
		Let $f \in \Lpm(\dotnorm_K)$. By definition, there exists an $\varepsilon>0$ with $\dotnorm_K+ \varepsilon f, \dotnorm_K- \varepsilon f \in \Subl(\R^n)$. Hence, there exist convex bodies $K_+$, $K_-$ with
		\begin{equation}
			h_{K_+} = \dotnorm_K +\varepsilon f \quad \text{and} \quad h_{K_-} = \dotnorm_K -\varepsilon f.
		\end{equation}
		Since $2 \dotnorm_K=2h_{K^\circ}$, we have $K_+ + K_- = 2K^\circ$ and hence $K_+, K_- \in \mathcal{S}(K^\circ)$. Moreover, we have
		\begin{equation}
			f= \frac{1}{2\varepsilon} \left(h_{K_+}-h_{K_-}\right)
		\end{equation}
		and hence $f \in \linh\iota(\mathcal{S}(K^\circ))$, as claimed.
		
		For the other inclusion, let $K_1, \dots, K_m \in \mathcal{S}(K^\circ)$, $\lambda_1,\dots,\lambda_m \in \R \setminus \{0\}$ and $f \coloneqq \sum_{i=1}^m \lambda_i h_{K_i}$. We have to show that $f \in \Lpm(\dotnorm_K)$. Let $i \in [m]$. Because $K_i \in \mathcal{S}(K^\circ)$, there exist a number $r_i>0$ and a convex body $L_i$ with $K^\circ = r_i K_i + L_i$. It follows that
		\begin{equation}
			\dotnorm_K - r_i h_{K_i} = h_{L_i} \in \Subl(\R^n).
		\end{equation}
		Without loss of generality, we assume that $\lambda_1,\dots,\lambda_{\ell}>0$ and $\lambda_{\ell+1},\dots,\lambda_m<0$ for some index $\ell \in \{0,\dots,m\}$. Let $\varepsilon \coloneqq \tfrac1m \min_{i \in [m]} \frac{r_i}{|\lambda_i|}$. Then we have
		\begin{equation}\label{eq_lemma_span_of_summands_1}
			\dotnorm_K + \varepsilon f = \dotnorm_K + \varepsilon \sum_{i=1}^m \lambda_i h_{K_i} = \frac\ell m \dotnorm_K +\varepsilon\sum_{i=1}^\ell  \lambda_i h_{K_i}+\sum_{i=\ell+1}^m \left(\frac1m \dotnorm_K-\varepsilon |\lambda_i| h_{K_i}\right).
		\end{equation}
		By our choice of $\varepsilon$, we have $\varepsilon |\lambda_i|\leq \frac{r_i}{m}$ and hence
		\begin{equation}
			\frac1m \dotnorm_K-\varepsilon |\lambda_i| h_{K_i} =\frac1m (\dotnorm_K-r_ih_{K_i})+\left(\frac{r_i}m-\varepsilon |\lambda_i|\right)h_{K_i}\in \Subl(\R^n)
		\end{equation}
		for all $i\in \{\ell+1,\dots,m\}$. Since the other summands on the right-hand side of \eqref{eq_lemma_span_of_summands_1} are sublinear as well, it follows that $\dotnorm_K + \varepsilon f \in \Subl(\R^n)$. The claim that $\dotnorm_K - \varepsilon f \in \Subl(\R^n)$ is shown analogously.
		
		The claim about the dimensions of the sets follows from the fact that the cone $\mathcal{S}(K^\circ)$ contains $\{0\}$.
	\end{proof}
	
	To make use of Proposition \ref{prop_L_K_second_derivative}, we need to construct admissible variations that witness the fact that certain convex bodies $K \in \Kno$ cannot be local maximizers of the isotropic constant. This is done in the following construction (see Figure \ref{fig_construction} for an illustration).
	
	\begin{construction} \label{constr_K_t}
		Let $K \in \Kno$ and let $f \in \Lpm(\dotnorm_K)$. Let $\varepsilon >0$ be small enough that
		\begin{enumerate}[label=(\roman*)]
			\item $\dotnorm_K+\varepsilon f$ and $\dotnorm_K-\varepsilon f$ are both sublinear, and
			\item \label{item_constr_K_t_1} we have $\dotnorm_K+\varepsilon f \geq \frac12 \dotnorm_K$ and $\dotnorm_K-\varepsilon f \geq \frac12 \dotnorm_K$.
		\end{enumerate}
		For $t \in [-\varepsilon,\varepsilon]$, let $K_t^f$ be the unit ball of the asymmetric norm given by
		\begin{equation}
			\norm{x}_{K_t^f} \coloneqq \norm{x}_{K} + t \cdot f(x) \quad \text{for } x\in \R^n.
		\end{equation}
		Since condition \ref{item_constr_K_t_1} ensures that $K_t^f$ is compact, $K_t^f$ is indeed a full-dimensional convex body for all $t \in [-\varepsilon,\varepsilon]$.
	\end{construction}
	
	\begin{figure}[ht]
		\begin{subfigure}[b]{.32\linewidth} 
			\centering
			\input{tikz/2/end1_of_movement.tex}
			
			\caption{start of movement $K_{-\varepsilon}$}
		\end{subfigure}
		\begin{subfigure}[b]{.33\linewidth}
			\centering
			\input{tikz/2/intermediate.tex}
			
			\caption{polytope $K=K_0$}
		\end{subfigure}
		\begin{subfigure}[b]{.32\linewidth}
			\centering
			\input{tikz/2/end2_of_movement.tex}
			
			\caption{end of movement $K_\varepsilon$}
		\end{subfigure}
		\caption{\label{fig_construction}An illustration of Construction \ref{constr_K_t} in the case where $K$ is a polygon. In Subfigure (A) and Subfigure (C), the sets $K_{-\varepsilon} \setminus K$ and $K_{\varepsilon} \setminus K$ are shaded in cyan, whereas the sets $K \setminus K_{-\varepsilon}$ and $K \setminus K_{\varepsilon}$ are shaded in orange. In this example, as in the polytopal case more generally, the vertices move on rays emanating from the origin (marked in blue); their distance to the origin at time $t$ is the weighted harmonic mean of the initial and the final distance.}
	\end{figure}
	
	We first verify that we have indeed constructed an admissible variation and compute the corresponding derivatives.
	
	\begin{lemma} \label{lemma_K_t_f_is_an_admissible_variation}
		Let $K \in \Kno$ and $f \in \Lpm(\dotnorm_K)$. Let $\varepsilon>0$ and $K_t \coloneqq K_t^f$ be as in Construction \ref{constr_K_t}. 
		Then $[-\varepsilon, \varepsilon] \rightarrow \Kno, t \mapsto K_t$ is an admissible variation of $K$ with 
		\begin{equation}
			\left.\ddt \int_{K_t} g(x) \dint x\right|_{t=0} = \int_{\sphere} -g(\rho(K,u)u) \cdot f(\rho(K,u)u) \cdot \rho(K,u)^{n} \dint u \quad \text{for } g \in C(\R^n).
		\end{equation}
		Moreover, if $f \neq 0$, we have
		\begin{equation} %
			\left.\ddtn{2}\int_{K_t} [\norm{x}_2^2-(n+2)] \dint x \right|_{t=0}>(n+3)\int_{\sphere}   f(\rho(K,u)u)^2 \left[ \rho(K,u)^{2} -(n+2)\right] \rho(K,u)^{n} \dint u.
		\end{equation}
		Here and in the following, integrals over $\sphere$ are with respect to the $(n-1)$-dimensional Hausdorff measure.
		
	\end{lemma}
	
	\begin{proof}
		For the sake of simplicity, we write
		\begin{equation}
			\rho_u(t) \coloneqq \rho(K_t,u) \quad \text{for } t \in [-\varepsilon,\varepsilon].
		\end{equation} 
		By construction, we have
		\begin{equation} \label{eq_rho_first_derivative}
			\rho'_u(0)= \left.\ddt \frac{1}{\norm{u}_{K_t}}\right|_{t=0}=\frac{-1}{\norm{u}_{K_0}^2} \cdot f(u)=-\rho_u(0)\cdot f(\rho_u(0)u)
		\end{equation}
		and
		\begin{equation} \label{eq_rho_second_derivative}
			\rho''_u(0)= \left.\ddt \frac{-f(u)}{\norm{u}_{K_t}^2}\right|_{t=0}=\frac{2}{\norm{u}_{K_0}^3} \cdot f(u)^2=2\rho_u(0)\cdot f(\rho_u(0)u)^2
		\end{equation}
		for all $u \in \sphere$.
		
		Let $g \in C(\R^n)$. By the Leibniz integral rule \cite[Thm.~2.27]{fol99}, we have
		\begin{align*}
			\left.\ddt \int_{K_t} g(x) \dint x\right|_{t=0} &= \left.\ddt \int_{\sphere} \int_0^{\rho_u(t)} g(ru) \cdot r^{n-1} \dint r \dint u \right|_{t=0} \\&=  \int_{\sphere}\left.\ddt \int_0^{\rho_u(t)} g(ru) \cdot r^{n-1} \dint r \right|_{t=0} \dint u \\
			 &=  \int_{\sphere}\rho_u'(0) \cdot g(\rho_u(0)u) \cdot \rho_u(0)^{n-1} \dint u\\&=  \int_{\sphere} -g(\rho_u(0)u) \cdot f(\rho_u(0)u) \cdot \rho_u(0)^{n} \dint u.
		\end{align*}
		Turning to the second derivatives, we apply the Leibniz integral rule twice and use \eqref{eq_rho_first_derivative} and \eqref{eq_rho_second_derivative} to obtain
		\begin{align}
			\left.\ddtn{2}\int_{K_t} \norm{x}_2^2\dint x \right|_{t=0} &= \left.\ddtn{2} \int_{\sphere}\int_0^{\rho_u(t)} r^2 \cdot r^{n-1} \dint r \dint u \right|_{t=0} \\
			&= \int_{\sphere} \left.\ddtn{2} \frac{\rho_u(t)^{n+2}}{n+2}\right|_{t=0} \dint u  \\&=  \int_{\sphere} \left((n+1)\cdot \rho_u(0)^{n} \cdot\rho'_u(0)^2 +\rho_u(0)^{n+1} \cdot\rho''_u(0)\right) \dint u\\
			&=(n+3)\int_{\sphere} f(\rho_u(0)u)^2 \cdot \rho_u(0)^{n+2}\dint u
		\end{align}
		and
		\begin{align}
			\left.\ddtn{2} \vol(K_t) \right|_{t=0} &= \left.\ddtn{2} \int_{\sphere}\int_0^{\rho_u(t)} r^{n-1} \dint r \dint u \right|_{t=0} \\
			&= \int_{\sphere} \left.\ddtn{2} \frac{\rho_u(t)^{n}}{n}\right|_{t=0} \dint u  \\&=  \int_{\sphere} \left((n-1)\cdot \rho_u(t)^{n-2} \cdot\rho'_u(t)^2 +\rho_u(t)^{n-1}\cdot \rho''_u(t)\right) \dint u\\
			&=(n+1)\int_{\sphere} f(\rho_u(0)u)^2 \cdot\rho_u(t)^{n} \dint u.
		\end{align}
		This shows that $t \mapsto K_t$ is indeed an admissible variation.
		
		With regard to the second claim, we have $f(\rho_u(0)u)^2 \cdot \rho_u(0)^{n}\geq 0$ for all $u \in \sphere$. Now let $f \neq 0$. Because $f$ is continuous, we even have $f(\rho_u(0)u)^2 \cdot \rho_u(0)^{n}> 0$ for all $u$ in some open subset of $\sphere$. This implies
		\begin{equation}
			\left.\ddtn{2} \vol(K_t) \right|_{t=0} <(n+3)\int_{\sphere} f(\rho_u(0)u)^2 \cdot\rho_u(t)^{n} \dint u.
		\end{equation}
		It follows that
		\begin{align} \label{eq_lemma_K_t_f_is_an_admissible_variation}
			&\left.\ddtn{2}\int_{K_t} [\norm{x}_2^2-(n+2)] \dint x \right|_{t=0}= 	\left.\ddtn{2}\int_{K_t} \norm{x}_2^2\dint x \right|_{t=0}-(n+2) \left.\ddtn{2} \vol(K_t) \right|_{t=0}\\&\quad\quad\quad>(n+3)\int_{\sphere}   f(\rho_u(0)u)^2 \left[ \rho_u(0)^{2} -(n+2)\right] \rho_u(0)^{n} \dint u. \qedhere
		\end{align}
	\end{proof}

	In order to prove Theorem \ref{thm_minkowski_summands}, we would like to show that the integral on the right-hand side of \eqref{eq_lemma_K_t_f_is_an_admissible_variation} is positive. For this, we assume towards a contradiction that $K$ is a local maximizer of the isotropic constant. If the integral on the right-hand side of \eqref{eq_lemma_K_t_f_is_an_admissible_variation} can be ``realized'' as the weak derivative of a suitable one-parameter family of convex bodies, then the sign of the expression under consideration can be determined by the first-order condition \eqref{eq_cor_first_order_condition_one_sided}. The following lemma shows that it is indeed possible to realize the integral on the right-hand side of \eqref{eq_lemma_K_t_f_is_an_admissible_variation} as a weak derivative.
	
	\begin{lemma} \label{lemma_f_squared_construction}
		Let $K \in \Kno$ and let $f \colon \R^n \rightarrow \R$ be a positively homogeneous function such that $\dotnorm_K+f$ and $\dotnorm_K-f$ are both convex. Then there exists an $\varepsilon>0$ such that
		\begin{equation}
			g_t \colon K \rightarrow \R, \quad g_t(x) \coloneqq \dotnorm_K+t f^2
		\end{equation}
		is convex for all $t \in [0,\varepsilon]$. For $t \in [0,\varepsilon]$, let $\tilde{K}_t$ be given by
		\begin{equation}
			\tilde{K}_t \coloneqq \{x \in K\mid g_t(x) \leq 1\}.
		\end{equation}
		Then $(\tilde{K}_t)_{t \in [0,\varepsilon]}$ is weakly differentiable at $t=0$ with
		\begin{equation} %
			\left.\ddt \int_{\tilde{K}_t} h(x) \dint x \right|_{t=0} =- \int_{\sphere}   f(\rho(K,u)u)^2 \cdot h(\rho(K,u) u) \cdot \rho(K,u)^{n} \dint u
		\end{equation}
		for all $h \in C(\R^n)$.
	\end{lemma}
	\begin{proof}
		We claim that $\varepsilon\coloneqq \left(1+2\max_{x \in K}|f(x)|\right)^{-1}$ has the stated property. Let $x,y \in K$ and let $z$ be given by $z=\lambda x+ (1-\lambda)y$ for $\lambda \in [0,1]$. If 
		\begin{equation}
			f(z)^2\leq \lambda f(x)^2 +(1-\lambda)f(y)^2,
		\end{equation}
		then the convexity of $\dotnorm_K$ implies
		\begin{equation}
			\norm{z}_K+t f(z)^2 \leq \lambda [\norm{x}_K+t f(x)^2]+ (1-\lambda) [\norm{y}_K+t f(y)^2] \quad \text{for } t \geq 0. %
		\end{equation}
		It remains to consider the case where 
		\begin{equation}\label{eq_f_squared_construction_interesting_case}
			f(z)^2> \lambda f(x)^2 +(1-\lambda)f(y)^2.
		\end{equation}
		Clearly, we have
		\begin{equation}
			(f(z)-f(z))^2-\lambda(f(x)-f(z))^2-(1-\lambda)(f(y)-f(z))^2\leq 0
		\end{equation}
		and hence, using \eqref{eq_f_squared_construction_interesting_case},
		\begin{align*}
			0<f(z)^2-\lambda f(x)^2 -(1-\lambda)f(y)^2 &\leq 2f(z) \cdot \left(f(z)-\lambda f(x)-(1-\lambda) f(y)\right)\\
			&= 2|f(z)| \cdot \sgn(f(z)) \cdot \left(f(z)-\lambda f(x)-(1-\lambda) f(y)\right).
		\end{align*}
		Therefore, for every $t \in [0,\varepsilon]$, we have
		\begin{align}
			t \cdot \left[f(z)^2-\lambda f(x)^2 -(1-\lambda)f(y)^2\right] &\leq \sgn(f(z)) \cdot (f(z)-\lambda f(x)-(1-\lambda) f(y)) \\&\leq  \lambda \norm{x}_K+(1-\lambda) \norm{y}_K-\norm{z}_K,
		\end{align}
		where the second step follows from the convexity of $\dotnorm_K+f$ and $\dotnorm_K-f$. Rearranging terms, we obtain the desired estimate
		\begin{equation}
			\norm{z}_K+t f(z)^2 \leq \lambda [\norm{x}_K+t f(x)^2]+ (1-\lambda) [\norm{y}_K+t f(y)^2] \quad \text{for } t \in [0,\varepsilon]. %
		\end{equation}
		
		With regard to the second claim, we observe that $\rho_u(t)\coloneqq \rho(\tilde{K}_t,u)$ satisfies the equation
		\begin{equation}
			1=\norm{\rho_u(t)u}_{\tilde{K_t}}= g_t(\rho_ut)=\norm{\rho_u(t) u}_K+t f(\rho_u(t) u)^2 = \rho_u(t) \cdot \norm{u}_K+t \cdot \rho_u(t)^2 \cdot f(u)^2 
		\end{equation}
		for all $u \in \sphere$ and $t \in [0,\varepsilon]$. Differentiating both sides yields
		\begin{equation}
			\rho'_u(t)\cdot\norm{ u}_K+t\cdot  \rho'_u(t) \cdot 2 \rho_u(t) \cdot f( u)^2 +f(\rho_u(t) u)^2 =0
		\end{equation}
		and we get
		\begin{equation}
			\rho'_u(0)= -\frac{f(\rho_u(0) u)^2}{\norm{ u}_K} = - \rho_u(0)\cdot f(\rho_u(0) u)^2.
		\end{equation}
		Let $h \in C(\R^n)$. By the Leibniz integral rule, we have
		\begin{align}
			\left.\ddt \int_{\tilde{K}_t}h(x) \dint x \right|_{t=0}  &= \int_{\sphere}  \rho'_u(0) \cdot h(\rho_u(0)u)\cdot \rho_u(0)^{n-1} 
			\!\dint u \\
			&= -\int_{\sphere}   f(\rho_u(0)u)^2 \cdot h(\rho_u(0)u) \cdot \rho_u(0)^{n} \dint u. \qedhere
		\end{align}
	\end{proof}
	For the proof of Theorem \ref{thm_minkowski_summands}, we now merely have to combine the previous results.
	\begin{proof}[Proof of Theorem \ref{thm_minkowski_summands}]
		Towards a contradiction, we assume that there exists a local maximizer $K \in \Kno$ of the isotropic constant with $\dim \mathcal{S}(K^\circ)> \frac{n^2+3n}{2}$. Because $\dim\mathcal{S}(K^\circ)$ is an affine invariant, we can assume without loss of generality that $K$ is isotropic. For each $f \in \Lpm(\dotnorm_K)$, let $\varepsilon_f>0$ and $K^f_t \colon [-\varepsilon_f,\varepsilon_f] \rightarrow \Knn$ be as in Construction \ref{constr_K_t}. %
		By Lemma \ref{lemma_K_t_f_is_an_admissible_variation}, the map
		\begin{equation}
			A \colon \Lpm(\dotnorm_K) \mapsto \Symn \times \R^n, \quad A(f) \coloneqq \left(\left[\left.\ddt \int_{K^f_t} x_i x_j \dint x\right|_{t=0}\right]_{i,j=1}^n,\left[\left.\ddt \int_{K^f_t} x_i \dint x\right|_{t=0}\right]_{i=1}^n\right)
		\end{equation}
		is linear, and from the fact that
		\begin{equation}
			\dim(\Symn \times \R^n) = \frac{n^2+n}{2}+n= \frac{n^2+3n}{2}
		\end{equation}
		it follows that $\ker A \neq \{0\}$. In other words, there exists a function $f \in \Lpm(\dotnorm_K) \setminus \{0\}$ such that $t \mapsto K_t \coloneqq K_t^f$ satisfies the assumptions of Proposition \ref{prop_L_K_second_derivative}. Possibly rescaling $f$, we can assume that $\varepsilon_f=1$. Using Lemma \ref{lemma_K_t_f_is_an_admissible_variation}, we get
		\begin{align} \label{eq_proof_of_thm_minkowski_summands}
			\left.\ddtn{2}L_{K_t}^{2n}\right|_{t=0}
			&=\frac{1}{\vol(K)^3}	\left.\ddtn{2} \int_{K_{t}}[\norm{x}_2^2-(n+2)] \dint x \right|_{t=0}\\&> \frac{n+3}{\vol(K)^3} \int_{\sphere}   f(\rho_u(0)u)^2 \left[ \rho_u(0)^{2} -(n+2)\right] \rho_u(0)^{n} \dint u.
		\end{align}
		Since $\varepsilon_f=1$, the function $f$ also satisfies the assumptions of Lemma \ref{lemma_f_squared_construction}. For $t \in [0,1]$, let $\tilde{K}_t$ be as in Lemma \ref{lemma_f_squared_construction}. Since $K$ is assumed to be a local maximizer of the isotropic constant, we have $\left.\ddt L_{\tilde{K}_t}^{2n}\right|_{t=0} \leq 0$ and hence, by Proposition \ref{prop_rademacher},
		\begin{align}
			\int_{\sphere}   f(\rho(K,u)u)^2 \left[ \rho(K,u)^{2} -(n+2)\right] \rho(K,u)^{n}  \dint u &= -\left.\ddt \int_{\tilde{K}_t}[\norm{x}_2^2-(n+2)] \dint x \right|_{t=0} \\&=- \vol(K)^3 \left.\ddt L_{\tilde{K}_t}^{2n}\right|_{t=0} \geq 0.
		\end{align}
		Combining this with \eqref{eq_proof_of_thm_minkowski_summands}, it follows that $\left.\ddtn{2}L_{K_t}^{2n}\right|_{t=0}>0$. But this is a contradiction to the assumption that $K$ is a local maximizer.
	\end{proof}
	
	\section{The polytopal case} \label{sect_the_polytopal_case}
	
	In this section, we study the polytopal case of Theorem \ref{thm_minkowski_summands}. Reconsidering the proof of Theorem \ref{thm_minkowski_summands}, we encounter the notion of facewise affine maps and discuss their relation to Minkowski decomposability. At the end of the section, we compare Theorem \ref{thm_minkowski_summands} with Corollary \ref{cor_generalized_RS_movements_dim}.
	
	Let $P \in \Kno$ be a polytope and $F \subset P$ a facet. What can be said about the structure of an arbitrary function $g \in \Lpm(\dotnorm_P)$? By definition, there exists an $\varepsilon>0$ such that
	$\dotnorm_P+ \varepsilon g$ and $\dotnorm_P- \varepsilon g$ are both sublinear. Since $\dotnorm_P$ is affine on $\pos F$, it follows that $\varepsilon g|_{\pos F}$ and $-\varepsilon g|_{\pos F}$ are both convex. In other words, $g|_{\pos F}$ has to be affine. This leads us to the following definition. Here and in the following, we write $\Phi_{n-1}(P)$ for the set of facets of $P$.
	
	\begin{definition}
		Let $P \subset \R^n$ be a polytope. A map $f \colon \bd P \rightarrow \R$ is called \emph{facewise affine} if $f|_F$ is affine for each $F \in \Phi_{n-1}(P)$. We denote the vector space of facewise affine maps on $\bd P$ by $\fwa(P)$.
	\end{definition}
	The considerations above show that the restriction $g|_{\bd P}$ of a function $g \in \Lpm(\dotnorm_P)$ is facewise affine. The converse statement is also true, i.e., if $P\in \Kno$, $\tilde{g} \in \fwa(P)$ and $g\colon \R^n \rightarrow \R$ is the unique positively homogeneous extension of $\tilde{g}$, then $g$ is an element of $\Lpm(\dotnorm_P)$. To see this, we observe that a sufficiently small radial perturbation leaves the vertices of $P$ in convex position. It follows that $\Lpm(\dotnorm_P)$ is isomorphic to $\fwa(P)$, and hence, a solid grasp of Theorem \ref{thm_minkowski_summands} is tantamount to a good understanding of the vector space $\fwa(P)$.
	
	In light of the considerations above, it is no surprise that the space $\fwa(P)$ appears in the study of Minkowski decomposability, notably in a paper by Smilansky \cite{smi87}. Let $V \coloneqq (v_1,v_2,\dots,v_m)$ be an $m$-tuple of vectors in $\R^n$. Following \cite{smi87}, we write
	\begin{equation}
		D(V) \coloneqq \left\{x \in \R^n \MID \textstyle\sum_{i=1}^mx_i v_i =0,\, \sum_{i=1}^mx_i=0\right\}
	\end{equation}
	and $D(P) \coloneqq D(\vertices P)$ for a polytope $P$. If $F \subset P$ is a face, then $D(F)$ is construed as a subspace of $D(P)$. Clearly, a facewise affine map $f \in \fwa(P)$ can be identified with its restriction $f|_{\vertices P}$. The defining constraints of being facewise affine are that the values  $f(v_1),\dots,f(v_m)$ of vertices $v_1,\dots,v_m$ on a single facet are compatible with the affine dependences of $v_1,\dots,v_m$. Combining this insight with the fact that $\fwa(P)$ and $\linh\mathcal{S}(P^\circ)$ are isomorphic, we recover a result by Smilansky \cite[Thm.~4.3]{smi87} (see also \cite[Thm.~11]{mcm73}), which asserts that
	\begin{equation} \label{eq_smilansky}
		\dim \mathcal{S}(P^\circ) = \# \vertices P - \dim \left[\textstyle\sum_{F\in\Phi_{n-1}(P)} D(F)\right].
	\end{equation}
	As remarked in \cite{smi87}, \eqref{eq_smilansky} shows that $\dim\mathcal{S}(P^\circ)$ can be determined by computing the kernel of a matrix $A$ whose rows encode the facetwise affine dependences of vertices of $P$, with linearly dependent constraints omitted. Such matrices have also been studied in the context of \emph{affine rigidity} \cite{zz09,gglt13}, i.e., the study of transformations of a point set $V \subset\R^n$ that preserve the affine dependences of certain subsets of $V$.
	
	Again, let $P \in \Kno$ be a polytope. By the considerations above, we have
	\begin{equation}
		n+1 \leq \dim \mathcal{S}(P^\circ) = \dim \fwa(P) \leq \# \vertices P,
	\end{equation}
	with equality in the second inequality if and only if $P$ is simplicial. The problem of characterizing equality in the lower bound is equivalent to the characterization of decomposable polytopes, which has received substantial attention. For an introduction to the topic, we refer to the survey by Shephard in Grünbaum's book \cite[Ch.~15]{grue03}. A more recent overview is given by Schneider \cite[Sect.~3.2]{sch13}. Some necessary conditions for equality in the lower bound can be stated in combinatorial terms (for example, it is known that $n+1 = \dim \mathcal{S}(P^\circ)$ holds if $P$ is simple), but in general, decomposability is not a combinatorial invariant. This was shown by Kallay \cite{kal82}, who also showed that $\dim\mathcal{S}(P)$ is a projective invariant \cite{kal84}. Therefore, $\dim\mathcal{S}(P^\circ)$ is a projective invariant as well.
	
	With regard to Theorem \ref{thm_minkowski_summands}, the decisive question is of course whether $\dim\mathcal{S}(P^\circ)$ exceeds the threshold value $\frac{n^2+3n}{2}$. If we want to use Theorem \ref{thm_minkowski_summands} to exclude certain combinatorial types of polytopes as potential local maximizers of $K \mapsto L_K$, we need a lower bound on $\dim\mathcal{S}(P^\circ)$ in terms of the combinatorial structure of $P$ or, equivalently, of $P^\circ$. This leads us to consider the following research problem.
	
	\begin{problem} \label{problem_combinatorial_lower_bound}
		Find a sharp lower bound on $\dim\mathcal{S}(Q)$ in terms of the combinatorial structure of a polytope $Q$. Is this lower bound attained by generic realizations of the combinatorial type of $Q$ (for a suitable meaning of the term ``generic'')?
	\end{problem}
	
	A thorough study of Problem \ref{problem_combinatorial_lower_bound} is beyond the scope of this paper. However, we already mentioned in Section \ref{sect_introduction} that the condition $\dim \fwa(P) \leq \frac{n^2+3n}{2}$ puts an upper bound on the number of connected components that we obtain by removing all facets that are simplices from the boundary complex of $P$. In the following, we will shortly pursue this line of thought. To see how simplices in the boundary of a polytope $P \in \Kno$ affect $\dim \fwa(P)$, we interpret the non-simplex facets of $P$ as hyperedges of a hypergraph. More precisely, we consider the hypergraph $H$ with vertex set $V(H) \coloneqq \vertices P$ and hyperedge set
	\begin{equation}
		E(H) \coloneqq \{\vertices F \mid F  \text{ is a facet of $P$, $F$ is not a simplex}\}.
	\end{equation}
	Let $\mathcal{C}(H) \subset 2^{V(H)}$ be the set of connected components of $H$. We define the dimension of a connected component $C \in \mathcal{C}(H)$ to be $0$ if $C$ consists of an isolated vertex, to be $n-1$ if $C$ consists of the vertices of a single facet, and $n$ otherwise. This definition agrees with $\dim \aff(C)$ for every realization of the combinatorial type of $P$. Since every collection of affine maps $f_C \colon \aff(C) \rightarrow \R$, $C \in \mathcal{C}(H)$, gives rise to a facewise affine map $f \in \fwa(P)$ with $f|_C=f_C|_C$ for $C \in \mathcal{C}(H)$, we have
	\begin{equation}
		\dim \fwa (P) \geq \sum_{C \in \mathcal{C}(H)} (\dim C+1).
	\end{equation}
	We recall that Theorem \ref{thm_simplicial_vertex} asserts that the hypergraph $H$ of a local maximizer $P$ of $K \mapsto L_K$ cannot have any isolated vertices unless $P$ is a simplex. Therefore, every connected component of $H$ has dimension $n$. Combining this with the condition $\dim \fwa(P) \leq \frac{n^2+3n}{2}$ from Theorem \ref{thm_minkowski_summands}, we obtain that the number of connected components of $H$ is of order $O(n)$. Arguably, our notion of a connected component is somewhat crude. One shortcoming is that the edge set $E(H)$ includes facets that are pyramids but not simplices, even though such a facet has no affine dependences that involve the apex of the pyramid. A more sophisticated definition of the hypergraph $E(H)$ might exclude as hyperedges all pyramidal faces of $P$, leading to a finer partition of $V$ into connected components. The problem with this approach is that the simple definition given above of the dimension of a connected component no longer works.
	
	\begin{remark}\label{rem_polytopal_case}\leavevmode
		\begin{enumerate}[label=(\roman*)]
			\item Interpreting the facets of a polytope $P \in \Kno$ as hyperedges of a hypergraph, we find ourselves precisely in the setting of the affine rigidity problem as considered in \cite{gglt13}. In \cite{gglt13}, it was shown that affine rigidity is a generic property of a hypergraph $H$, i.e., it is independent of the embedding of $H$ into $\R^n$ as long as the coordinates of the embedded vertices have algebraically independent coordinates. That $\dim \fwa(P)$ is not a combinatorial property of the hypergraph induced by $P$ is due to the fact that the vertices of $P$ cannot have algebraically independent coordinates, except in the trivial case where $P$ is simplicial.
			\item There are further well-studied vector spaces that are isomorphic to $\linh \mathcal{S}(P^\circ)$ and $\fwa (P)$, including the following:
			\begin{enumerate}
				\item the space of parallel redrawings of the (geometric) edge graph of $P^\circ$,
				\item the space of parallel redrawings of the hyperplane arrangement induced by the facets of $P^\circ$.
			\end{enumerate}
			The fact that these vector spaces are isomorphic creates a connection between the topics of Minkowski decomposability, (affine) rigidity and scene analysis; for more information, we refer to \cite{whi89,sw17}.
		\end{enumerate}
	\end{remark}
	
	To conclude this section, we analyze how Theorem \ref{thm_minkowski_summands} relates to Corollary \ref{cor_generalized_RS_movements_dim}. For this, let $P \in \Kno$ be a polytope and let $\beta$ be the speed function of a generalized RS-movement of $P$ in direction $u \in \sphere$. Let $F\subset P$ be a facet. We say that $F$ is \emph{vertical} with respect to $u$ if the outer unit normal vector $v$ of $F$ is orthogonal to $u$. Let $F$ be non-vertical with respect to $u$ with, say, $\scpr{v,u} > 0$. 
	Because $\beta$ is the speed function of a generalized RS-movement, there exists an $\varepsilon>0$ such that $g_{P,u}+\varepsilon \beta^+$ and $g_{P,u}-\varepsilon \beta^+$ are both convex on $\pi_{u^\bot}(F)$. Since $g_{P,u}|_{\pi_{u^\bot}(F)}$ is affine, it follows that $\beta^+|_{\pi_{u^\bot}(F)}$ is affine as well. 
	This shows that the vector space $\GRS(P,u)$ from Definition \ref{def_GRS} is closely related to $\fwa(P)$. 
	To compare the logical strengths of Theorem \ref{thm_minkowski_summands} and Corollary \ref{cor_generalized_RS_movements_dim}, we have to compare the dimensions of $\GRS(P,u)$ and $\fwa(P)$. Here we have to distinguish four cases.
	\begin{enumerate}[label=(\roman*)]
		\item $P$ does not have vertical facets with respect to $u$.
		\begin{enumerate}
			\item We have $\pi_{u^\bot}(P)=P\cap u^\bot$. In this case, the map $\GRS(P,u) \rightarrow \fwa(P)$, $\beta \mapsto \beta|_{\bd P}$ is an isomorphism, and we have $\dim \GRS(P,u) = \dim \mathcal{S}(P^\circ)$.
			\item We have $\pi_{u^\bot}(P)\neq P\cap u^\bot$. In this case, the map $\GRS(P,u) \rightarrow \fwa(P)$, $\beta \mapsto \beta|_{\bd P}$ is well-defined, but not surjective. Therefore, we have $\dim \GRS(P,u) < \dim \mathcal{S}(P^\circ)$.
		\end{enumerate}
		\item $P$ has at least one vertical facet with respect to $u$.
		\begin{enumerate}
			\item We have $\pi_{u^\bot}(P)=P\cap u^\bot$. Here, the map $\fwa(P) \rightarrow \GRS(P,u)$ that interpolates affinely on each chord $P \cap (x+\linh \{u\})$, is well-defined and linear, but possibly not surjective. It follows that $\dim \GRS(P,u) \geq \dim \mathcal{S}(P^\circ)$.
			\item We have $\pi_{u^\bot}(P)\neq P\cap u^\bot$. In this case, there exist facewise affine maps that do not correspond to valid speed functions, and possibly speed functions that are not facewise affine. The situation is indeterminate.
		\end{enumerate}
	\end{enumerate}
	
	\section{Symmetric convex bodies} \label{sect_symmetric_convex_bodies}
	
	In this section, we discuss a variant of Theorem \ref{thm_minkowski_summands} for classes of convex bodies whose symmetry group contains a given subgroup of $\On$. In particular, we discuss the special cases of centrally symmetric, unconditional and 1-symmetric convex bodies.
	
	Let $K \subset \R^n$ be a convex body. An isometry $U \colon \R^n \rightarrow \R^n$ is called a \emph{symmetry} of $K$ if it satisfies $U(K) = K$. The symmetries of $K$ form a group, we call it the \emph{symmetry group} of $K$ and denote it by $G(K)$. We note in passing that some authors use the term ``symmetry group'' for the group of orientation-preserving symmetries of $K$ and call the group $G(K)$ the ``full symmetry group'' of $K$. Every $U \in G(K)$ has $c_K$ as a fixed point. In particular, if $K$ is centered, then all symmetries of $K$ are linear and $G(K)$ is a subgroup of $\On$. For a function $f \colon \R^n \rightarrow \R$, we define a \emph{symmetry} of $f$ to be an isometry $U \colon \R^n \rightarrow \R^n$ with $f \circ U = f$. As in the case of a convex body, the symmetries of a function form a group, which we call the \emph{symmetry group} of $f$ and denote by $G(f)$. Clearly, the symmetry group of a convex body $K$ coincides with the symmetry groups of its characteristic function $\One_K$, its support function $h_K$ and (if $K \in \Kno$) its gauge function $\dotnorm_K$.
	
	\begin{definition}
		Let $G$ be a subgroup of $\On$. We say that a convex body $K \subset \R^n$ or a function $f \colon \R^n \rightarrow \R$ is \emph{$G$-symmetric} if its symmetry group contains $G$. We denote by $\mathcal{K}^n_G \subset \Knn$ the space of full-dimensional $G$-symmetric convex bodies. For $K \in \mathcal{K}_G^n$, we write $\mathcal{S}_G(K)\coloneqq \mathcal{S}(K) \cap \mathcal{K}_G^n$.
	\end{definition}
	
	The results in this section are based on the simple fact that the covariance matrix of a $G$-symmetric convex body $K$ is invariant under conjugation by elements of $G \subset \On$, i.e., that we have
	\begin{equation}
		A_K = A_{U(K)}= U A_K U^\transp \quad \text{for all } U \in G.
	\end{equation}
	
	\begin{remark}
		It is natural to ask whether a more natural concept of symmetry in our setting is given by \emph{affine symmetries} of $K$, i.e., affine isomorphisms $f \colon \R^n \rightarrow \R^n$ with $f(K)=K$, which also form a group under composition. However, if $K$ is isotropic, then every affine symmetry $f$ of $K$ has to be linear, and the fact that the covariance matrix of $K$ is invariant under conjugation by $f$ implies that $f$ is a linear isometry. In other words, for isotropic convex bodies, our notion of symmetry group coincides with the group of affine symmetries of $K$.
	\end{remark}
	
	We recall that the \emph{centralizer} of a subgroup $G$ of the general linear group $\GLnR$ is given by
	\begin{equation}
		{\textstyle\C_{\GLnR}}(G) = \{A \in \GLnR \mid UA = AU \text{ for all } U \in G\}.
	\end{equation}
	Since we are dealing with covariance matrices, which are symmetric, we are interested in the vector space
	\begin{equation}
		V_G \coloneqq {\textstyle\C_{\GLnR}(G)} \cap \Symn.
	\end{equation}
	We also have to keep track of the centroid of $K$. For this, we consider the vector space of fixed points of $G$
	\begin{equation}
		W_G \coloneqq \{x \in \R^n \mid Ux=x \text{ for all } U \in G\}.
	\end{equation}
	The main result of this section asserts that if we restrict our attention to $G$-symmetric convex bodies, then the upper bound for the dimension of decomposability depends on the dimensions of $V_G$ and $W_G$.
	
	\begin{theorem} \label{thm_symmetries}
		Let $G$ be a subgroup of $O(n)$ and let $K \subset \mathcal{K}^n_G$ be a local maximizer of the isotropic constant in the class $\mathcal{K}^n_G$. Then we have
		\begin{equation}
			\dim \mathcal{S}_G(K^\circ) \leq \dim V_G + \dim W_G.
		\end{equation}
	\end{theorem}
	
	\begin{proof}
		The proof is essentially a verbatim repetition of the arguments from Section \ref{sect_minkowski_decomposability}. Defining $\Lpm_G(f)$ as the set that contains the $G$-symmetric elements of $\Lpm(f)$, we have $\dim \mathcal{S}_G(K^\circ)=\Lpm_G(\dotnorm_K)$. For $f \in \Lpm_G(\dotnorm_K)$, the one-parameter families $K_t^f$ and $\tilde{K}_t^f$ from Construction \ref{constr_K_t} and Lemma \ref{lemma_f_squared_construction} are both contained in $\mathcal{K}^n_G$. Finally, repeating the proof of Theorem \ref{thm_minkowski_summands}, we first note that we can still assume without loss of generality that $K$ is isotropic, even though $\mathcal{K}_G^n$ might not contain an isotropic image of $K$. Indeed, if $M\in \GLnR$ has the property that $\tilde{K}\coloneqq M(K-c_{K})$ is isotropic, then for every $U \in G$ we have
		\begin{equation}
			MUM^{-1} \tilde{K}=MUM^{-1}M(K-c_K)=M(K-c_K)=\tilde{K}.
		\end{equation}
		In other words, we have $MGM^{-1}\subset G(\tilde{K})$, and we can replace $K$ by $\tilde{K}$ and $G$ by $MGM^{-1}$ in the proof without affecting the dimensions of the vector spaces under consideration. Next, we observe that the image of the map
		\begin{equation}
			A \colon \Lpm_G(\dotnorm_K) \mapsto \Symn \times \R^n, \quad A(f) \coloneqq \left(\left[\left.\ddt \int_{K^f_t} x_i x_j \dint x\right|_{t=0}\right]_{i,j=1}^n,\left[\left.\ddt \int_{K^f_t} x_i \dint x\right|_{t=0}\right]_{i=1}^n\right)
		\end{equation}
		is contained in $V_G \times W_G$. Therefore, the assumption $\dim \Lpm_G(\dotnorm_K) > \dim V_G + \dim W_G$ is sufficient to derive that $A$ has a non-trivial kernel, which yields a contradiction.
	\end{proof}
	
	In the following, we spell out the implications of Theorem \ref{thm_symmetries} for three nested classes of convex bodies. The first example we consider is given by $o$-symmetric convex bodies. In this case, the group $G$ is generated by $-\id$, and we have $V_G= \Symn$ and $W_G=\{0\}$.
	
	\begin{corollary}
		Let $K \subset \R^n$ be an $o$-symmetric full-dimensional convex body that is a local maximizer of $K \mapsto L_K$ in the class of centrally symmetric convex bodies. Then
		\begin{equation}
			\dim \mathcal{S}_{\langle-\id\rangle}(K^\circ) \leq \frac{n^2+n}{2}.
		\end{equation}
	\end{corollary}
	
	For $i \in \{1,\dots,n\}$, let $R_i$ denote the reflection across the $i$-th coordinate hyperplane. A convex body that is invariant under $R_i$ for all $i \in \{1,\dots,n\}$ is called \emph{unconditional}. If $G={\langle R_1,\dots,R_n\rangle}$ is generated by the reflections across the coordinate hyperplanes, then again we have $W_G=\{0\}$, but this time the vector space $V_G$ only consists of diagonal matrices. We obtain the following necessary condition for local maximizers of $K \mapsto L_K$ in the class of unconditional bodies.
	
	\begin{corollary}
		Let $K \subset \R^n$ be an unconditional full-dimensional convex body that is a local maximizer of $K \mapsto L_K$ in the class of unconditional convex bodies. Then
		\begin{equation}
			\dim \mathcal{S}_{\langle R_1,\dots,R_n\rangle}(K^\circ) \leq n.
		\end{equation}
	\end{corollary}
	
	Finally, let $G_1 \subset \On$ be the group that is generated by $R_1,\dots,R_n$ and by all permutations of coordinates of $\R^n$. A $G_1$-symmetric convex body is called \emph{1-symmetric}. Since $V_{G_1}= \linh \{I_n\}$ and $W_G=\{0\}$, Theorem \ref{thm_symmetries} implies that a local maximizer of $K \mapsto L_K$ in the class of 1-symmetric convex bodies cannot allow any non-trivial decompositions of its gauge function. Again, we record this fact as a corollary.
	
	\begin{corollary}
		Let $K \subset \R^n$ be a 1-symmetric full-dimensional convex body that is a local maximizer of $K \mapsto L_K$ in the class of 1-symmetric convex bodies. Then $K$ is indecomposable in the class of 1-symmetric convex bodies, i.e., we have
		\begin{equation}
			\dim \mathcal{S}_{G_1}(K^\circ) =1.
		\end{equation}
	\end{corollary}
	
	\section{Concluding remarks} \label{sect_concluding_remarks}
	
	Considering our main result Theorem \ref{thm_minkowski_summands}, it is natural to ask whether the bound on $\dim \mathcal{S}(K^\circ)$ can be improved. As discussed in the introduction, not much is known about local maximizers of the isotropic constant. The only well-understood setting seems to be the planar case. Here it is known that triangles are the only local maximizers, as shown by Saroglou \cite{sar10}. In other words, every local maximizer $K \in \mathcal{K}_{o}^2$ of $K\mapsto L_K$ satisfies
	\begin{equation}
		\dim\mathcal{S}(K^\circ) = 3 < 5 =\frac{2^2+3\cdot 2}{2}.
	\end{equation}
	In this sense, Theorem \ref{thm_minkowski_summands} already fails to be optimal in the plane. In dimensions $n \geq 3$, no local maximizers seem to be known (in particular, it seems to be unknown whether simplices are actually local maximizers). Therefore, it is conceivable that every local maximizer $K \in \Kno$ of $K\mapsto L_K$ has an indecomposable polar body $K^\circ$, i.e., it satisfies
	\begin{equation}\label{eq_improved_bound}
		\dim\mathcal{S}(K^\circ) = n+1 < \frac{n^2+3n}{2}.
	\end{equation}
	We note in passing that \eqref{eq_improved_bound} would imply Theorem \ref{thm_simplicial_vertex} as well as the analogous result about cubical zonotopes. Clearly, the indecomposability of $K^\circ$ would follow if the map $K \mapsto L_{K^\circ}$ were quasi-convex with respect to Minkowski addition. But this is false; a counter-example is given by two non-homothetic ellipsoids $E_1$, $E_2$, for which we have
	\begin{equation}
		L_{(\frac12(E_1+E_2))^\circ}>L_{E_1}=L_{E_2}.
	\end{equation}%
	We note, however, that the failure of $K \mapsto L_{K^\circ}$ to be quasi-convex does not imply the existence of decomposable local maximizers -- indeed, such local maximizers do not exist in the planar case even though $K \mapsto L_{K^\circ}$ fails to be quasi-convex in this setting. %
	In higher dimensions, it is of course conceivable that there exist decomposable local maximizers -- for all we know, such local maximizers might even have dimensions of decomposability of order $\Omega(n^2)$.
	
	Aside from quantitative improvements to Theorem \ref{thm_minkowski_summands}, a natural line of further thought is to look for analogues in terms of other notions of decomposability, the most obvious candidate being the Minkowski decomposability of $K$. It is by no means clear how the proof strategy from Section \ref{sect_minkowski_decomposability} might be adapted to this problem, and the representation in terms of support functions seems to be less suited for the task of computing the integrals at hand. As in the polar case, it is natural to ask whether $K \mapsto L_{K}$ might be quasi-convex with respect to Minkowski addition, and again, the answer is negative, as witnessed by the example of two non-homothetic ellipsoids $E_1$, $E_2$. %
	Nevertheless, it is an interesting question whether local maximizers of $K\mapsto L_K$ can be Minkowski decomposable, and if so, what the maximal dimension of decomposability might be.
	
	\subsection*{Acknowledgments}
	
	The research was done at TU Berlin. Funded by the Deutsche Forschungsgemeinschaft (DFG, German Research Foundation) under Germany's Excellence Strategy – The Berlin Mathematics Research Center MATH+ (EXC-2046/1, project ID: 390685689). I am indebted to Ansgar Freyer, Emanuel Milman and the anonymous reviewers for their careful reading of this manuscript and many helpful comments.%
	
	\subsection*{Data availability}
	
	No data was used for the research described in the article.

	\printbibliography
\end{document}

%% file: tikz/1/beta_and_K.tex
\begin{tikzpicture}%
	[scale=1.000000,
	axis/.style={dotted, thin},
	axis2/.style={loosely dotted, thin},
	edge/.style={color=black},
	edge2/.style={color=red,dashed},
	red/.style={color=red},
	orange/.style={fill=orange,fill opacity=0.200000},
	cyan/.style={fill=cyan,fill opacity=0.200000},
	vertex/.style={},
	facet3/.style={color=blue, opacity=0.50000},
	facet4/.style={fill=red,fill opacity=0.20000}]

	\draw[-stealth] (0,0) -- (0,1);

	\draw[edge] (-2,0.5) -- (-1,2);
	
	\draw[edge] (-1,2) -- (1,1.5);
	\draw[edge,axis] (-1,2) -- (-1,4);
	
	\draw[edge] (1,1.5) -- (2,0);
	\draw[edge,axis] (1,1.5) -- (1,4);
	
	\draw[edge] (2,0) -- (1.25,-1.5);
	\draw[edge,axis] (2,0) -- (2,4);
	\draw[edge,axis] (2,0) -- (2,-3.5);
	
	\draw[edge] (1.25,-1.5) -- (-1,-1.5);
	\draw[edge,axis] (1.25,-1.5) -- (1.25,-3.5);
	
	\draw[edge] (-1,-1.5) -- (-2,0.5);
	\draw[edge,axis] (-1,-1.5) -- (-1,-3.5);
	
	\node at (0.75,-0.75) {$K$};
	
	\node at (0.25,0.5) {\footnotesize $u$};
	
	\draw[facet3] (-1.75,0) -- (2,0);
	\draw[red] (-1.75,0) -- (-2,0);
	
	\draw[edge2] (-1.75,-3.5) -- (-1.75,4);
	\draw[edge2] (-2,-3.5) -- (-2,4);
	\fill[facet4] (-2,-3.5) -- (-2,4) -- (-1.75,4) --(-1.75,-3.5)  -- (-2,-3.5);

	\node at (-2.3,3.5) {$\beta^+$};
	\draw[edge,axis] (-2,4) -- (2,4);
	\node at (2.3,4) {\footnotesize $+1$};
	\draw[edge] (-2,3.5) -- (2,3.5);
	\node at (2.3,3.5) {\footnotesize $0$};
	\draw[edge,axis] (-2,3) -- (2,3);
	\node at (2.3,3) {\footnotesize $-1$};
	
	\node at (-2.3,-3) {$\beta^-$};
	\draw[edge,axis] (-2,-3.5) -- (2,-3.5);
	\node at (2.3,-3.5) {\footnotesize $-1$};
	\draw[edge] (-2,-3) -- (2,-3);
	\node at (2.3,-3) {\footnotesize $0$};
	\draw[edge,axis] (-2,-2.5) -- (2,-2.5);
	\node at (2.3,-2.5) {\footnotesize $+1$};
	
	\fill[cyan] (-2, 3.5+0.25) -- (-1, 3.5+0.5) -- (-1, 3.5) -- (-2, 3.5) -- cycle {};
	\draw[edge] (-2, 3.5+0.25) -- (-1, 3.5+0.5);
	
	\fill[orange] (1, 3.5-0.3) -- (-1+2*5/8, 3.5) -- (1, 3.5) -- cycle {};
	\fill[cyan] (-1, 3.5+0.5) -- (-1+2*5/8, 3.5) -- (-1, 3.5) -- cycle {};
	\draw[edge] (-1, 3.5+0.5) -- (1, 3.5-0.3);
	
	\fill[cyan] (2, 3.5+0.2) -- (1+5/8, 3.5+0) -- (2, 3.5+0) -- cycle {};
	\fill[orange] (1, 3.5-0.3) -- (1+5/8, 3.5+0) -- (1, 3.5+0) -- cycle {};
	\draw[edge] (1, 3.5-0.3) -- (2, 3.5+0.2);
	
	\fill[orange] (-2, -3+0.25) -- (-1, -3+0.5) -- (-1, -3) -- (-2, -3) -- cycle {};
	\draw[edge] (-2, -3+0.25) -- (-1, -3+0.5);

	\fill[orange] (-1, -3+0.5) -- (-1, -3) -- (-1+2.25/2,-3) -- cycle {};
	\fill[cyan] (-1+2.25/2,-3) -- (1.25, -3-0.5) -- (1.25, -3) -- cycle {};
	\draw[edge] (-1, -3+0.5) -- (1.25, -3-0.5);
	
	\fill[orange] (2, -3+0.2) --  (1.25+5/7*3/4, -3+0) -- (2, -3+0) -- cycle {};
	\fill[cyan] (1.25, -3-0.5) -- (1.25+5/7*3/4, -3+0) -- (1.25, -3+0) -- cycle {};
	\draw[edge] (1.25, -3-0.5) -- (2, -3+0.2);
	
\end{tikzpicture}

%% file: tikz/1/K01.tex
\begin{tikzpicture}%
	[scale=1.000000,
	back/.style={loosely dotted, thin},
	edge/.style={color=black},
	facet2/.style={fill=cyan,fill opacity=0.200000},
	facet1/.style={fill=orange,fill opacity=0.200000},
	facet/.style={fill=white,fill opacity=1.00000},
	vertex/.style={}]
	
	\coordinate (2.00000, 0.00000) at (2.00000, 0.00000);
	\coordinate (1.00000, -1.50000) at (1.00000, -1.50000);
	\coordinate (-0.50000, -1.50000) at (-0.50000, -1.50000);
	\coordinate (1.00000, 1.50000) at (1.00000, 1.50000);
	\coordinate (-2.00000, 0.00000) at (-2.00000, 0.00000);
	\coordinate (-1.00000, 2.00000) at (-1.00000, 2.00000);
	\fill[facet1]  (-1.00000, -1.50000) -- (-2.00000, 0.50000) -- (-1.00000, 2.00000) -- (1.00000, 1.50000) -- (2.00000, 0.00000) -- (1.25000, -1.50000) -- cycle {};
	\fill[facet2](-1.00000, -1.40000) -- (-2.00000, 0.55000) -- (-1.00000, 2.10000) -- (1.00000, 1.44000) -- (2.00000, 0.04000) -- (1.25000, -1.60000) -- cycle {};
	\fill[facet] (0.12500, -1.50000) -- (1.25000, -1.50000) -- (1.78571, -0.42857) -- (1.98915, 0.01627) -- (1.60000, 0.60000) -- (1.00000, 1.44000) -- (0.25000, 1.68750) -- (-1.00000, 2.00000) -- (-1.98551, 0.52174) -- (-1.00000, -1.40000) -- cycle {};
	
	\draw[edge] (-2.00000, 0.55000) -- (-1.00000, 2.10000);
	\draw[edge] (-2.00000, 0.55000) -- (-1.00000, -1.40000);
	\draw[edge] (-1.00000, 2.10000) -- (1.00000, 1.44000);
	\draw[edge] (1.00000, 1.44000) -- (2.00000, 0.04000);
	\draw[edge] (2.00000, 0.04000) -- (1.25000, -1.60000);
	\draw[edge] (1.25000, -1.60000) -- (-1.00000, -1.40000);
	
	\draw[edge] (-2.00000, 0.50000) -- (-1.00000, 2.00000);
	\draw[edge] (-2.00000, 0.50000) -- (-1.00000, -1.50000);
	\draw[edge] (-1.00000, 2.00000) -- (1.00000, 1.50000);
	\draw[edge] (1.00000, 1.50000) -- (2.00000, 0.00000);
	\draw[edge] (2.00000, 0.00000) -- (1.25000, -1.50000);
	\draw[edge] (1.25000, -1.50000) -- (-1.00000, -1.50000);
	
	\node at (0,0) {$K_{0.1} \triangle K$};
	
\end{tikzpicture}

%% file: tikz/1/K02.tex
\begin{tikzpicture}%
	[scale=1.000000,
	back/.style={loosely dotted, thin},
	edge/.style={color=black},
	facet2/.style={fill=cyan,fill opacity=0.200000},
	facet1/.style={fill=orange,fill opacity=0.200000},
	facet/.style={fill=white,fill opacity=1.00000},
	vertex/.style={}]
	
	\coordinate (2.00000, 0.00000) at (2.00000, 0.00000);
	\coordinate (1.00000, -1.50000) at (1.00000, -1.50000);
	\coordinate (-0.50000, -1.50000) at (-0.50000, -1.50000);
	\coordinate (1.00000, 1.50000) at (1.00000, 1.50000);
	\coordinate (-2.00000, 0.00000) at (-2.00000, 0.00000);
	\coordinate (-1.00000, 2.00000) at (-1.00000, 2.00000);
	\fill[facet1] (-1.00000, -1.50000) -- (-2.00000, 0.50000) -- (-1.00000, 2.00000) -- (1.00000, 1.50000) -- (2.00000, 0.00000) -- (1.25000, -1.50000) -- cycle {};
	\fill[facet2] (-1.00000, -1.30000) -- (-2.00000, 0.60000) -- (-1.00000, 2.20000) -- (1.00000, 1.38000) -- (2.00000, 0.08000) -- (1.25000, -1.70000) -- cycle {};
	
	\fill[facet] (0.12500, -1.50000) -- (1.25000, -1.50000) -- (1.78571, -0.42857) -- (1.97935, 0.03098) -- (1.60000, 0.60000) -- (1.00000, 1.38000) -- (0.25000, 1.68750) -- (-1.00000, 2.00000) -- (-1.97059, 0.54412) -- (-1.00000, -1.30000) -- cycle {};
	\draw[edge] (-2.00000, 0.60000) -- (-1.00000, 2.20000);
	\draw[edge] (-2.00000, 0.60000) -- (-1.00000, -1.30000);
	\draw[edge] (-1.00000, 2.20000) -- (1.00000, 1.38000);
	\draw[edge] (1.00000, 1.38000) -- (2.00000, 0.08000);
	\draw[edge] (2.00000, 0.08000) -- (1.25000, -1.70000);
	\draw[edge] (1.25000, -1.70000) -- (-1.00000, -1.30000);
	
	\draw[edge] (-2.00000, 0.50000) -- (-1.00000, 2.00000);
	\draw[edge] (-2.00000, 0.50000) -- (-1.00000, -1.50000);
	\draw[edge] (-1.00000, 2.00000) -- (1.00000, 1.50000);
	\draw[edge] (1.00000, 1.50000) -- (2.00000, 0.00000);
	\draw[edge] (2.00000, 0.00000) -- (1.25000, -1.50000);
	\draw[edge] (1.25000, -1.50000) -- (-1.00000, -1.50000);

	\node at (0,0) {$K_{0.2} \triangle K$};
\end{tikzpicture}

%% file: tikz/2/end1_of_movement.tex
\begin{tikzpicture}%
	[scale=0.3500000,
	back/.style={loosely dotted, thin},
	edge/.style={color=black},
	facet2/.style={fill=cyan,fill opacity=0.200000},
	facet1/.style={fill=orange,fill opacity=0.200000},
	facet/.style={fill=white,fill opacity=1.00000},
	vertex/.style={},
	vertex2/.style={inner sep=1pt,circle,draw=darkgray!25!black,fill=darkgray!75!black,thick},
	vertex3/.style={inner sep=1pt,circle,draw=darkgray!25!black},
	vertex4/.style={inner sep=1pt,circle,draw=blue,fill=blue},
	axis/.style={dotted, thin}]
	
	\coordinate (2.00000, 0.00000) at (2.00000, 0.00000);
	\coordinate (1.00000, -1.50000) at (1.00000, -1.50000);
	\coordinate (-0.50000, -1.50000) at (-0.50000, -1.50000);
	\coordinate (1.00000, 1.50000) at (1.00000, 1.50000);
	\coordinate (-2.00000, 0.00000) at (-2.00000, 0.00000);
	\coordinate (-1.00000, 2.00000) at (-1.00000, 2.00000);
	\fill[facet1]  (-5.45455, -3.27273) -- (0.00000, -5.68421) -- (4.50000, -3.00000) -- (4.88889, 3.55556) -- (-0.88889, 5.33333) -- (-4.94118, 0.82353) -- cycle {};
	\fill[facet2] (4.40000, 3.20000) -- (-1.00000, 6.00000) -- (-6.00000, 1.00000) -- (-5.00000, -3.00000) -- (0.00000, -6.00000) -- (4.50000, -3.00000) -- cycle {};
	\fill[facet] (-5.28000, -1.88000) -- (-5.00000, -3.00000) -- (-2.00000, -4.80000) -- (0.00000, -5.68421) -- (4.50000, -3.00000) -- (4.40000, 3.20000) -- (2.00000, 4.44444) -- (-0.88889, 5.33333) -- (-4.94118, 0.82353) -- cycle {};
	
	\draw[edge] (-6.00000, 1.00000) -- (-5.00000, -3.00000);
	\draw[edge] (-6.00000, 1.00000) -- (-1.00000, 6.00000);
	\draw[edge] (-5.00000, -3.00000) -- (0.00000, -6.00000);
	\draw[edge] (-1.00000, 6.00000) -- (4.40000, 3.20000);
	\draw[edge] (0.00000, -6.00000) -- (4.50000, -3.00000);
	\draw[edge] (4.50000, -3.00000) -- (4.40000, 3.20000);
	
	\draw[edge] (4.88889, 3.55556) -- (4.50000, -3.00000);
	\draw[edge] (4.88889, 3.55556) -- (-0.88889, 5.33333);
	\draw[edge] (4.50000, -3.00000) -- (0.00000, -5.68421);
	\draw[edge] (0.00000, -5.68421) -- (-5.45455, -3.27273);
	\draw[edge] (-0.88889, 5.33333) -- (-4.94118, 0.82353);
	\draw[edge] (-4.94118, 0.82353) -- (-5.45455, -3.27273);
	\draw[axis] (0, 0) -- (-6.00000, 1.00000);
	\draw[axis] (0, 0) -- (-6.00000, -3.60000);
	\draw[axis] (0, 0) -- (-1.00000, 6.00000);
	\draw[axis] (0, 0) -- (0.00000, -5.40000);
	\draw[axis] (0, 0) -- (4.50000, -3.00000);
	\draw[axis] (0, 0) -- (5.50000, 4.00000);
	
	\node[vertex3] at (-4.20000, 0.70000)     {};
	\node[vertex3] at (-6.00000, -3.60000)     {};
	\node[vertex3] at (-0.80000, 4.80000)     {};
	\node[vertex3] at (0.00000, -5.40000)     {};
	\node[vertex3] at (4.50000, -3.00000)     {};
	\node[vertex2] at (4.40000, 3.20000)     {};
	
	\node[vertex2] at (-6.00000, 1.00000)     {};
	\node[vertex2] at (-5.00000, -3.00000)     {};
	\node[vertex2] at (-1.00000, 6.00000)     {};
	\node[vertex2] at (0.00000, -6.00000)     {};
	\node[vertex2] at (4.50000, -3.00000)     {};
	\node[vertex3] at (5.50000, 4.00000)     {};

	\node[vertex4] at (0, 0)     {};

\end{tikzpicture}

%% file: tikz/2/intermediate.tex
\begin{tikzpicture}%
	[scale=0.3500000,
	back/.style={loosely dotted, thin},
	edge/.style={color=black},
	facet2/.style={fill=cyan,fill opacity=0.200000},
	facet1/.style={fill=orange,fill opacity=0.200000},
	facet/.style={fill=white,fill opacity=1.00000},
	vertex/.style={},
	vertex2/.style={inner sep=1pt,circle,draw=darkgray!25!black,fill=darkgray!75!black,thick},
	vertex3/.style={inner sep=1pt,circle,draw=darkgray!25!black},
	vertex4/.style={inner sep=1pt,circle,draw=blue,fill=blue},
	axis/.style={dotted, thin}]
	
	\coordinate (2.00000, 0.00000) at (2.00000, 0.00000);
	\coordinate (1.00000, -1.50000) at (1.00000, -1.50000);
	\coordinate (-0.50000, -1.50000) at (-0.50000, -1.50000);
	\coordinate (1.00000, 1.50000) at (1.00000, 1.50000);
	\coordinate (-2.00000, 0.00000) at (-2.00000, 0.00000);
	\coordinate (-1.00000, 2.00000) at (-1.00000, 2.00000);
	
	\draw[edge] (4.88889, 3.55556) -- (4.50000, -3.00000);
	\draw[edge] (4.88889, 3.55556) -- (-0.88889, 5.33333);
	\draw[edge] (4.50000, -3.00000) -- (0.00000, -5.68421);
	\draw[edge] (0.00000, -5.68421) -- (-5.45455, -3.27273);
	\draw[edge] (-0.88889, 5.33333) -- (-4.94118, 0.82353);
	\draw[edge] (-4.94118, 0.82353) -- (-5.45455, -3.27273);
	\node[vertex3] at (-4.20000, 0.70000)     {};
	\node[vertex3] at (-6.00000, -3.60000)     {};
	\node[vertex3] at (-0.80000, 4.80000)     {};
	\node[vertex3] at (0.00000, -5.40000)     {};
	\node[vertex3] at (4.50000, -3.00000)     {};
	\node[vertex3] at (4.40000, 3.20000)     {};
	
	\node[vertex3] at (-6.00000, 1.00000)     {};
	\node[vertex3] at (-5.00000, -3.00000)     {};
	\node[vertex3] at (-1.00000, 6.00000)     {};
	\node[vertex3] at (0.00000, -6.00000)     {};
	\node[vertex3] at (4.50000, -3.00000)     {};
	\node[vertex3] at (5.50000, 4.00000)     {};
	
	\node[vertex2] at (4.88889, 3.55556)     {};
	\node[vertex2] at (4.50000, -3.00000)     {};
	\node[vertex2] at (0.00000, -5.68421)     {};
	\node[vertex2] at (-0.88889, 5.33333)     {};
	\node[vertex2] at (-4.94118, 0.82353)     {};
	\node[vertex2] at (-5.45455, -3.27273)     {};
	
	\node[vertex4] at (0, 0)     {};
	
	\draw[axis] (0, 0) -- (-6.00000, 1.00000);
	\draw[axis] (0, 0) -- (-6.00000, -3.60000);
	\draw[axis] (0, 0) -- (-1.00000, 6.00000);
	\draw[axis] (0, 0) -- (0.00000, -5.40000);
	\draw[axis] (0, 0) -- (4.50000, -3.00000);
	\draw[axis] (0, 0) -- (5.50000, 4.00000);

\end{tikzpicture}

%% file: tikz/2/end2_of_movement.tex
\begin{tikzpicture}%
	[scale=0.350000,
	back/.style={loosely dotted, thin},
	edge/.style={color=black},
	facet2/.style={fill=cyan,fill opacity=0.200000},
	facet1/.style={fill=orange,fill opacity=0.200000},
	facet/.style={fill=white,fill opacity=1.00000},
	vertex/.style={},
	vertex2/.style={inner sep=1pt,circle,draw=darkgray!25!black,fill=darkgray!75!black,thick},
	vertex3/.style={inner sep=1pt,circle,draw=darkgray!25!black},
	vertex4/.style={inner sep=1pt,circle,draw=blue,fill=blue},
	axis/.style={dotted, thin}]
	
	\coordinate (2.00000, 0.00000) at (2.00000, 0.00000);
	\coordinate (1.00000, -1.50000) at (1.00000, -1.50000);
	\coordinate (-0.50000, -1.50000) at (-0.50000, -1.50000);
	\coordinate (1.00000, 1.50000) at (1.00000, 1.50000);
	\coordinate (-2.00000, 0.00000) at (-2.00000, 0.00000);
	\coordinate (-1.00000, 2.00000) at (-1.00000, 2.00000);
	\fill[facet1] (-5.45455, -3.27273) -- (0.00000, -5.68421) -- (4.50000, -3.00000) -- (4.88889, 3.55556) -- (-0.88889, 5.33333) -- (-4.94118, 0.82353) -- cycle {};
	\fill[facet2] (5.50000, 4.00000) -- (-0.80000, 4.80000) -- (-4.20000, 0.70000) -- (-6.00000, -3.60000) -- (0.00000, -5.40000) -- (4.50000, -3.00000) -- cycle {};
	\fill[facet] (-5.45455, -3.27273) -- (-2.00000, -4.80000) -- (0.00000, -5.40000) -- (4.50000, -3.00000) -- (4.88889, 3.55556) -- (2.00000, 4.44444) -- (-0.80000, 4.80000) -- (-4.20000, 0.70000) -- (-5.28000, -1.88000) -- cycle {};
	
	\draw[edge] (-4.20000, 0.70000) -- (-6.00000, -3.60000);
	\draw[edge] (-4.20000, 0.70000) -- (-0.80000, 4.80000);
	\draw[edge] (-6.00000, -3.60000) -- (0.00000, -5.40000);
	\draw[edge] (-0.80000, 4.80000) -- (5.50000, 4.00000);
	\draw[edge] (0.00000, -5.40000) -- (4.50000, -3.00000);
	\draw[edge] (4.50000, -3.00000) -- (5.50000, 4.00000);
	
	\draw[edge] (4.88889, 3.55556) -- (4.50000, -3.00000);
	\draw[edge] (4.88889, 3.55556) -- (-0.88889, 5.33333);
	\draw[edge] (4.50000, -3.00000) -- (0.00000, -5.68421);
	\draw[edge] (0.00000, -5.68421) -- (-5.45455, -3.27273);
	\draw[edge] (-0.88889, 5.33333) -- (-4.94118, 0.82353);
	\draw[edge] (-4.94118, 0.82353) -- (-5.45455, -3.27273);
	\draw[axis] (0, 0) -- (-6.00000, 1.00000);
	\draw[axis] (0, 0) -- (-6.00000, -3.60000);
	\draw[axis] (0, 0) -- (-1.00000, 6.00000);
	\draw[axis] (0, 0) -- (0.00000, -5.40000);
	\draw[axis] (0, 0) -- (4.50000, -3.00000);
	\draw[axis] (0, 0) -- (5.50000, 4.00000);
	
	\node[vertex2] at (-4.20000, 0.70000)     {};
	\node[vertex2] at (-6.00000, -3.60000)     {};
	\node[vertex2] at (-0.80000, 4.80000)     {};
	\node[vertex2] at (0.00000, -5.40000)     {};
	\node[vertex2] at (4.50000, -3.00000)     {};
	\node[vertex3] at (4.40000, 3.20000)     {};
	
	\node[vertex3] at (-6.00000, 1.00000)     {};
	\node[vertex3] at (-5.00000, -3.00000)     {};
	\node[vertex3] at (-1.00000, 6.00000)     {};
	\node[vertex3] at (0.00000, -6.00000)     {};
	\node[vertex3] at (4.50000, -3.00000)     {};
	\node[vertex2] at (5.50000, 4.00000)     {};

	\node[vertex4] at (0, 0)     {};
	
\end{tikzpicture}